\documentclass[a4paper]{article}

\usepackage{amssymb,amsmath,amsthm,verbatim}
\usepackage[dvips,final]{graphicx}
\usepackage{psfrag}
\usepackage{hyperref}
\usepackage{enumitem}
\setlist{nosep}
\setlist[enumerate,1]{label={\rm (\roman*)}}
\setlist[1]{labelindent=\parindent}
\usepackage{natbib}

\usepackage{color}

\newtheorem{theorem}{Theorem}[section]
\newtheorem{lemma}[theorem]{Lemma}
\newtheorem{cor}[theorem]{Corollary}
\newtheorem{prop}[theorem]{Proposition}

\newtheorem{definition}[theorem]{Definition}

\newtheorem{remark}[theorem]{Remark}

\renewcommand {\theequation}{\arabic{section}.\arabic{equation}}

\newcommand{\re}{\mathbb{R}}
\newcommand{\R}{\mathcal{R}}
\newcommand{\D}{\mathcal{D}}

\newcommand{\p}{\mathbb{P}}
\newcommand{\supp}{\mbox{supp}}

\newcommand{\e}{\mathbb{E}}

\newcommand{\F}{\mathcal{F}}

\newcommand{\Ss}{\mathcal{S}}

\newcommand{\ol}[1]{\overline{#1}}
\newcommand{\me}{\mathrm{e}}
\newcommand{\td}{\mathrm{d}}
\newcommand{\T}{\mathcal{T}}
\newcommand{\Ij}{\mathrm{I}_j}
\newcommand{\ovl}{\overline}

\newcommand{\Fc}{\mathcal{F}}

\def \1{{\bf 1}}

\newcommand{\pmu}[1]{U^{\mu_{#1}}}

\newcommand{\bmu}{\pmb{\mu}}

\newcommand{\Eps}[2]{\e^{#1}\left[#2\right]}
\newcommand{\Epsb}[3]{\e^{#1}_{#2}\!\left[#3\right]}
\newcommand{\Epsba}[3]{\e^{#1}_{#2}\!\left|#3\right|}
\newcommand{\Epsa}[2]{\e^{#1}\left|#2\right|}

\newcommand{\pps}[2]{\p^{#1}\left[#2\right]}

\newcommand{\eps}{\varepsilon}

\newcommand{\half}{\frac{1}{2}}
\newcommand{\Lc}{\mathcal{L}}

\newcommand{\Tx}{T_{\xi}}

\newcommand{\sx}{\sigma^\xi}
\newcommand{\sxb}{\sigma^{\xi^\beta}}
\newcommand{\srb}{\sigma_{\R^\beta}}
\newcommand{\ax}{{\alpha^\xi}}
\newcommand{\xb}{{\xi^\beta}}
\newcommand{\sib}{\sigma^\beta}
\newcommand{\tb}{\bar{t}}

\newcommand{\dU}{\delta U}
\newcommand{\du}{\delta u}

\newcommand{\I}{\mathcal{I}}
\newcommand{\bI}{{\bar\I}}
\newcommand{\iI}{\I^{\circ}}

\begin{document}

\title{\sc The Root solution to the multi--marginal embedding problem: an optimal stopping and time--reversal approach\thanks{We are grateful to many who have commented on this research project. In particular, we thank the 
participants of the BIRS Workshop \emph{Mathematical Finance: Arbitrage and Portfolio Optimization} in May 2014 and HIM Workshop \emph{Optimal Transport and Stochastics} in March 2015 for their helpful comments and remarks.}
}
\author{
Alexander M.~G.~Cox\thanks{University of Bath, email:
  \href{mailto:A.M.G.Cox@bath.ac.uk}{\nolinkurl{A.M.G.Cox@bath.ac.uk}}, web: \url{http://www.maths.bath.ac.uk/\~mapamgc/}}
\and Jan Ob\l{}\'{o}j\thanks{University of Oxford, email: \href{mailto:jan.obloj@maths.ox.ac.uk}{\nolinkurl{jan.obloj@maths.ox.ac.uk}}, web: \url{http://www.maths.ox.ac.uk/people/jan.obloj}. The research has received funding from the European Research Council under the European Union's Seventh Framework Programme (FP7/2007-2013) / ERC grant agreement no. 335421. The author is also grateful to the Oxford-Man Institute of Quantitative Finance and St John's College in Oxford for their support.}
 \and Nizar Touzi\thanks{Ecole Polytechnique Paris \href{mailto:nizar.touzi@polytechnique.edu}{\nolinkurl{nizar.touzi@polytechnique.edu}}. The research has received funding from the European Research Council under the European Union's Seventh Framework Programme (FP7/2007-2013) / ERC grant agreement no. 321111. The author also gratefully acknowledges the financial support from the Chair {\it Financial Risks} of the {\it Risk Foundation} sponsored by Soci\'et\'e G\'en\'erale, and the Chair {\it Finance and Sustainable Development} sponsored by EDF and CA-CIB.}}

\maketitle

\begin{abstract} 
We provide a complete characterisation of the Root solution to the Skorokhod embedding problem (SEP) by means of an optimal stopping formulation. Our methods are purely probabilistic and the analysis relies on a tailored time-reversal argument. This approach allows us to address the long-standing question of a multiple marginals extension of the Root solution of the SEP. Our main result establishes a complete solution to the $n$--marginal SEP using first hitting times of barrier sets by the time-space process. The barriers are characterised by means of a recursive sequence of optimal stopping problems. Moreover, we prove that our solution enjoys a global optimality property extending the one-marginal Root case. Our results hold for general, one-dimensional, martingale diffusions.  
\end{abstract}

\section{Introduction}

The Skorokhod embedding problem (SEP) for Brownian motion $(B_t)_{t \ge 0}$ consists of specifying a stopping time $\sigma$ such that $B_\sigma$ is distributed according to a given probability measure $\mu$ on $\re$. It has been an active field of study in probability since the original paper by \cite{Skorokhod:65}, see \cite{Obloj:04b} for an account. One of the most natural ideas for a solution is to consider $\sigma$ as the first hitting time of some \emph{shape} in time--space. This was carried out in an elegant paper of \cite{Root}. Root showed that for any centred and square integrable distribution $\mu$ there exists a \emph{barrier} $\R$, i.e. a subset of $\re_+\times\re$ such that $(t,x)\in \R$ implies $(s,x)\in \R$ for all $s\geq t$, for which $B_{\sigma_{\R}}\sim \mu$, $\sigma_{\R}=\inf\{t: (t,B_t)\in \R\}$. The barrier is (essentially) unique, as argued by \cite{Loynes}.

Root's solution enjoys a fundamental optimality property, established by \cite{Rost}, that it minimises the variance of the stopping time among all solutions to the SEP. More generally, $\e f(\sigma_{\R})\leq \e f(\sigma)$ for any convex function $f\geq 0$ and any stopping time $\sigma$ with $B_\sigma\sim B_{\sigma_{\R}}$. 
This led to a recent revival of interest in this construction in the mathematical finance literature, where optimal solutions to SEP are linked to robust pricing and hedging of derivatives, see \cite{Hobson:98,Hobson:11}. More precisely, optimality of the Root solution translates into lower bounds on prices of options written on the realised volatility.

In recent work \cite{CoxWang:11} show that the barrier $\R$ may be written as the unique solution to a Free Boundary Problem (FBP) or, more generally, to a Variational Inequality (VI). This yields directly its representation by means of an optimal stopping problem. This observation was the starting point for our study here. Subsequently, \cite{GassiatOberhauserdosReis} used analytic methods based on the theory of viscosity solutions to extend Root's existence result to the case of general, integrable starting and target measures satisfying the convex ordering condition. Using methods from optimal transport, \cite{BeiglbockCoxHuesmann} have also recently proved the existence and optimality of Root solutions for one-dimensional Feller processes and, under suitable assumptions on the target measure, for Brownian motion in higher dimensions.

The first contribution of our paper is to show that one can obtain the barrier $\R$ directly from the optimal stopping formulation, and to prove the embedding property using purely probabilistic methods. This also allows us to determine a number of interesting properties of $\R$ by means of a time-reversal technique. Our results will hold for a general one-dimensional diffusion.

Beyond the conceptual interest in deriving the Root solution from the optimal stopping formulation, the new perspective enables us to address the long--standing question of extending the Root solution of the Skorokhod embedding problem to the multiple-marginals case, i.e.~given a non-decreasing (in convex order) family of $n$ probability measures $(\mu_0,\ldots,\mu_n)$ on $\re$ with finite first moment, and a diffusion $X$ started from the measure $\mu_0$, find stopping times $\sigma_1\le\ldots\le\sigma_n$ such that $X_{\sigma_i}\sim\mu_i$, and $X_{.\wedge\sigma_n}$ is uniformly integrable. Our second contribution, and the main result of the paper, provides a complete characterisation of such a solution to the SEP which extends the Root solution in the sense that it enjoys the following two properties:
\begin{itemize}
\item first, the stopping times are defined as hitting times of a sequence of barriers, which are completely characterized by means of a recursive sequence of optimal stopping problems;
\item second, similar to the one-marginal case, we prove that our solution of the multiple marginal SEP minimizes the expectation of any non-decreasing convex function of $\rho_n$ among all families of stopping times $\rho_1\le\ldots\le\rho_n$, such that $X_{\rho_i} \sim \mu_i$. 
\end{itemize}
It is well known that solutions to the multiple marginal SEP exist if and only if the measures are in convex order, however finding optimal solutions to the multiple marginal SEP is more difficult.  While many classical constructions of solutions to embedding problems can, in special cases, be ordered (see \cite{MadanYor}), in general the ordering condition is not satisfied except under strong conditions on the measures. The first paper to produce optimal solutions to the multiple marginal SEP was \cite{BHR}, who extended the single marginal construction of \cite{AzemaYor} to the case where one intermediate marginal is specified. More recently, \cite{OblojSpoida} and \cite{HLOST} extended these results to give an optimal construction for an arbitrary sequence of $n$ marginals satisfying a mild technical condition.

There are also a number of papers which make explicit connections between optimal stopping problems and solutions to the SEP, including \cite{Peskir:99}, \cite{Obloj:04} and \cite{Cox:2008aa}. In these papers, the key observation is that the optimal solution to the SEP can be closely connected to a particular optimal stopping problem; in all these papers, the \emph{same} stopping time gives rise to both the optimal solution to the SEP, and the optimal solution to a related optimal stopping problem. In this paper, we will see that the key connection is not that the same stopping time solves both the SEP and a related optimal stopping problem, but rather that there is a time-reversed optimal stopping problem which has the same stopping region as the SEP, and moreover, the value function of the optimal stopping problem has a natural interpretation in the SEP. The first paper we are aware of to exploit this connection is \cite{McConnell}, who works in the setting of the solution of \cite{Rost:71} and \cite{Chacon} to the SEP (see also \cite{CoxWang:13,GassiatOberhauserdosReis}), and uses analytic methods to show that Rost's solution to the SEP has a corresponding optimal stopping interpretation. More recently\footnote{Indeed, we were made aware of this paper only in the final stages of completing this work.} \cite{DeAngelis:15} has provided a probabilistic approach to understanding McConnell's connection, using a careful analysis of the differentiability of the value function to deduce the embedding properties of the SEP; both the papers of McConnell and De~Angelis also require some regularity assumptions on the underlying measures in order to establish their results. In contrast, we consider the Root solution to the SEP. As noted above, a purely analytic connection between Root's solutions to the SEP and a related (time-reversed) optimal stopping problem was observed in \cite{CoxWang:11}. In this paper, we are not only able to establish the embedding problems based on properties of the related optimal stopping problem, but we are also able to use our methods to prove new results (in this case, the extension to multiple marginal solutions, and characterisation of the corresponding stopping regions), without requiring any assumptions on the measures which we embed (beyond the usual convex ordering condition).

The paper is organized as follows. Section \ref{sect:Root} formulates the multiple marginals Skorokhod embedding problem, reviews the Root solution together with the corresponding variational formulation, and states our optimal stopping characterization of the Root barrier. In Section \ref{sect:multipleRoot}, we report the main characterisation of the multiple marginal solution of the SEP, and we derive the corresponding optimality property. The rest of the paper is devoted to the proof of the main results. In Section \ref{sect:induction}, we introduce some important definitions relating to potentials, state the main technical results, and use these to prove our main result regarding the embedding properties. The connection with optimal stopping is examined in Section \ref{sect:stoppedpotential}. Given this preparation, we report the proof of the main result in Section \ref{sect:finiteatoms} in the case of locally finitely supported measures. This is obtained by means of a time reversal argument. Finally, we complete the proof in the case of general measures in Section \ref{sect:approxatom} by a delicate limiting procedure. 

\vspace{5mm}

\noindent {\bf Notation and Standing Assumptions:} In the following, we consider a regular, time-homogenous, martingale diffusion taking values on an interval $\I$, defined on a filtered probability space $(\Omega, \F, (\F_t), \p)$ satisfying the usual hypotheses. For $(t,x) \in \re_+ \times \re$, we write $\e^{t,x}$ for expectations under the measure for which the diffusion departs from $x$ at time $t$. We also write $\e^{x}=\e^{0,x}$. We use both $(X_t)$ and $(Y_t)$ to denote the diffusion process. While $X$ and $Y$ denote the same object, the double notation allows us to distinguish between two interpretations: with a fixed reference time-space domain $\re_+ \times \re$, we think of $(X_t)$ as starting in $(t,x)$ and \emph{running forward in time} and of $(Y_t)$ as starting in $(t,x)$ and \emph{running backwards in time}. For a distribution $\nu$ on $\re$, we interpret $\e^{\nu}[.] = \int  \e^{x}[.]\nu(dx)$.

We suppose that the diffusion coefficient is $\eta(x)$, so $d\langle X\rangle_t = \eta(X_t)^2 dt$, where $\eta$ is locally Lipschitz, $|\eta(x)|^2 \le C_\eta (1+|x|^2)$, for some constant $C_\eta$, and strictly positive on $\I^\circ$, where we write $\I^\circ = (a_{\I}, b_{\I})$, and without loss of generality, assume that $0 \in \I^{\circ}$; in addition, we use $\bI$ for the closure of $\I$, and $\partial \bI$ for the boundary, so $\partial \bI = \{a_\I, b_{\I}\}$. We assume that the corresponding endpoints are either absorbing (in which case they are in $\I$), or inaccessible (in which case, if for example $b_{\I}$ is inaccessible and finite, then $\p(X_t \to b_{\I} \text{ as } t \to \infty) >0$). The measures we wish to embed will be assumed to be supported on $\bar\I$, and in the case where $\I \neq \bar\I$, it may be possible to embed mass at $\partial \bI$ by taking a stopping time which takes the value $\infty$. We define $\Ss := [0,\infty] \times \bI$. We note also\footnote{See the proof of Lemma \ref{lem:vxi-immediate} below for a suitable argument.} that as a consequence of the assumption on $\eta$, we have $\Eps{x}{X_t^2} < \infty$, and we further write $m_{\mu_0}(t) := \Epsa{\mu_0}{X_t}$ for suitable measures $\mu_0$.

We will also frequently want to restart the space-time process, given some stopped distribution in both time and space, and we will write $\xi$ for a general probability measure on $\Ss$, with typically $\xi \sim (\sigma,X_{\sigma})$ for some stopping time $\sigma$. With this notation, we have, $\Eps{\xi}{A} = \int \Eps{t,x}{A}\xi(dt,dx)$ and we denote $(\Tx,X_{\Tx})$ the random starting point, which then has law $\xi$. Since $\xi$ may put mass on $\bI \setminus \I$, we interpret the process started at such a point as the constant process. For each of these processes, $L_t^x$ denotes the (semimartingale) local time at $x$ corresponding to the process $X_t$, with the convention that $L_{t}^x = 0$ for $t \le \Tx$. In addition, given a barrier $\R$, we define the corresponding hitting time of $\R$ by $X$ under $\p^\xi$ by:
\begin{equation*}
  \sigma_\R = \inf \{ t \ge \Tx: (t,X_{t}) \in \R\}.
\end{equation*}
Similarly, given a stopping time $\sigma_0$ we write
\begin{equation*}
  \sigma_{\R}(\sigma_0) = \inf \{ t \ge \sigma_0: (t,X_{t}) \in \R\}.
\end{equation*}

Finally, we observe that, as a consequence of the (local) Lipschitz property of $\eta$, we know there exists a continuous transition density, $p: (0,\infty) \times \I^{\circ} \times \I^{\circ}$, so that
\begin{equation*}
  \e^{x}\left[ f(X_t)\right] = \int p(t,x,y) f(y) \, dy,
\end{equation*}
whenever $f$ is supported in $\I$ (see e.g. \cite[Theorem~V.50.11]{Rogers:2000aa}). We observe that we then have the following useful identities for the local time (see e.g. \cite[Theorem~3.7.1]{KaratzasShreve}):
\begin{equation*}
  \int_0^t f(X_s) \eta^2(X_s)\, ds = \int f(a) L_t^a \, da
\end{equation*}
and 
\begin{equation}\label{eq:1}
  \e^y[L_t^x] = \eta(x)^2 \int_0^t p(s,y,x) \, ds.
\end{equation}

\section{The Root solution of the Skorokhod embedding problem}
\label{sect:Root}
\setcounter{equation}{0}

\subsection{Definitions}

Throughout this paper, we consider a sequence of centred
probability measures $\bmu_n:=(\mu_i)_{i=0,\dots,n}$ on $\bI$:
 \begin{eqnarray}\label{eq:2a}
  \quad \int_{\I} |x|\mu_i(dx)<\infty,
  &\mbox{and}&
  \int_{\I} x \mu_i(dx)=0,
  ~~i=0,\ldots,n.
  \end{eqnarray}
We similarly denote $\bmu_k = (\mu_0,\mu_1, \dots, \mu_k)$ for all $k\le n$. We say that $\bmu_k$ is in convex order, and we denote $\mu_0\preceq_{\text{\rm cx}}\ldots\preceq_{\text{\rm cx}}\mu_k$, if
  \begin{eqnarray}\label{eq:2b}
  \int_{\re} c(x) \mu_{i-1}(dx) 
  \le
  \int_{\re} c(x) \mu_{i}(dx),
  i=1,\ldots,k
  &\mbox{for all convex functions}&
  c.
\end{eqnarray}
The lower and the upper bounds of the support of $\mu_k$ \emph{relative to $\mu_{k-1}$} are denoted by
 \begin{eqnarray}\label{eq:ellDefn}
 \ell_k
 :=
 \inf\big\{x:\mu_k\big[(-\infty,x)\big] \neq \mu_{k-1}\big[(-\infty,x)\big]\big\}
 &\mbox{and}&
 r_k:=\sup\big\{x:\mu_k\big[(x,\infty)\big]\neq\mu_{k-1}\big[(x,\infty)\big]\big\}.\quad
 \end{eqnarray}
 We exclude the case where $\mu_k = \mu_{k-1}$ as a trivial special case,
 and so we always have $\ell_k < r_k$ for all $k=1,\ldots,n$, as a consequence of the convex ordering.  The potential of a probability measure $\mu$ is defined by
 \begin{eqnarray}
 U^{\mu}(x)
 &:=&
 - \int_{\re} |x-y|\mu(\td y);
 ~~x\in\re,
 \end{eqnarray}
see \cite{Chacon:77}. For centred measures $\bmu_n$ in convex order, we have 
 \begin{eqnarray}\label{UmuleU0}
 U^{\mu_{k}} \le U^{\mu_{k-1}}
 &\mbox{and}&
 U^{\mu_k}=U^{\mu_{k-1}}~~\mbox{on}~~(\ell_k,r_k)^c,
 ~~\mbox{for all}~~k=1,\ldots,n.
 \end{eqnarray}

Recall that $(X_t)_{t \in \re_+}$ is a martingale diffusion. A stopping time $\sigma$ (which may take the value $\infty$ with positive probability) is said to be uniformly integrable (UI) if the process $(X_{t \wedge \sigma})_{t \ge 0}$ is uniformly integrable under $\p^{\mu_0}$. We denote by $\T$ the collection of all UI stopping times.

The classical Skorokhod embedding problem with starting measure $\mu_0$ and target measure $\mu_1$ is:
 \begin{eqnarray}\label{SEPmu}
 {\rm SEP}(\bmu_1):\quad\quad \textrm{find }\sigma\in\T\textrm{ such that }
 X_\sigma \sim \mu_1\textrm{ under }\p^{\mu_0}.\hspace*{2.9cm}
 \end{eqnarray}

 We consider the problem with multiple marginals: 
 \begin{eqnarray}\label{eq:nSEP}
 {\rm SEP}(\bmu_n):\quad\quad
 \textrm{find } 0\le\sigma_1\ldots\le\sigma_n\in\T\textrm{ such that }
 X_{\sigma_k} \sim \mu_k,~k=1,\ldots,n \textrm{ under }\p^{\mu_0}.
 \end{eqnarray}
 In this paper, our interest is in a generalisation of the \cite{Root} solution of the Skorokhod embedding problem so that each stopping time $\sigma_k$ is the first hitting time, after $\sigma_{k-1}$, by $(t,X_t)_{t\ge 0}$ of some subset $\R$ in $\Ss$. Further, and crucially, we require that $\R$ is a barrier in the following sense:
 \begin{definition}\label{def:barrier}
A set $\R \subset \Ss$ is called a \emph{barrier} if
\\
$\bullet$ $\R$ is closed;
\\
$\bullet$ if $(t,x)\in \R$ then $(s,x)\in \R$ for all $s\geq t$;
\\
$\bullet$ if $x \in \{a_{\I}, b_{\I}\}$ is finite, $(0,x) \in \R$.
\\
Given a barrier $\R$, for $x\in \bI$, we define the corresponding barrier function:
\begin{eqnarray}\label{eq:tbar_def}
\ovl t_\R(x)
&:=&
\inf\{t\geq 0: (t,x)\in \R\}\in [0,\infty].
\end{eqnarray} 
\end{definition}

Since $\R$ is closed it follows, as observed by \cite{Root} and \cite{Loynes}, that $\ovl t_\R (\cdot)$ is lower semi--continuous on $\I$. Also, from the second property, we see that a barrier is the epigraph of the corresponding barrier function in the $(t,x)$-plane:
\begin{eqnarray*}
\R
&=&
\big\{(t,x)\in\re_+\times\I:~t\ge\ovl t_\R(x)\big\}.
\end{eqnarray*}

\begin{definition}\label{def:regularbarrier}

{\rm (i)} We say that a barrier is \emph{regular} if $\{x\in \iI: \ovl t_\R(x)>0\}$ is an open interval containing zero. 
\\
{\rm (ii)} For a probability measure $\xi = \xi(dt,dx)$ on $\Ss$, we say that a barrier is \emph{$\xi$-regular} if 
\begin{eqnarray*}
  \p^\xi\big[\sigma_\R = \sigma_{\R^{(t,x)}}\big] <1
  &\mbox{for all}&
  (t,x) \not \in \R,
  ~~\mbox{where}~~
  \R^{(t,x)} = \R \cup \left( [t,\infty) \times \{x\}\right),
\end{eqnarray*}
i.e. the barrier cannot be enlarged without altering the stopping distribution of the space-time diffusion started with law $\xi$ and run to the hitting of $\R$. 
\end{definition}

Observe that a regular barrier is a $\delta_{(0,0)}$-regular barrier. We have the following characterisation:
\begin{remark}\label{rem:regular lt}
A barrier $\R$ is $\xi$-regular  if and only if $\Eps{\xi}{L_{t\wedge \sigma_\R}^x} < \Eps{\xi}{L_{\sigma_\R}^x}$ for all $(t,x) \not\in \R$. 
\end{remark}
\begin{lemma}\label{lem:unique}
Let $\xi$ be a probability measure on $\Ss$ and $\R$ a barrier such that $\inf_{x \in \I} \ovl t_\R(x)<\infty$. Then $\sigma_\R<\infty$ or $\lim_{t \to \infty} X_t \in \{a_{\I},b_{\I}\}$ $\p^\xi$-a.s. Further, if $\R$ is not $\xi$-regular then there exists a $\xi$-regular barrier $\tilde\R\supseteq\R$ such that $X_{\sigma_\R}\sim X_{\sigma_{\tilde \R}}$ $\p^\xi$-a.s.
\end{lemma}
\begin{proof}
For some $x_0\in \I$, we have $\ovl t_\R(x_0)<\infty$ and  $\{(t,x): t\geq \ovl t_\R(x_0)\}\subset \R$.
If $\I = \re$ then $\limsup_t X_t = \infty$ and $\liminf X_t = -\infty$ and it is clear that $\sigma_\R<\infty$ $\p^\xi$-a.s. Otherwise $\lim_{t \to \infty} X_t \in \{a_{\I}, b_{\I}\}$ . If $\R$ is not $\xi$-regular then by definition the set of all barriers $\tilde \R$ for which $X_{\sigma_\R}\sim X_{\sigma_{\tilde \R}}$ $\p^\xi$-a.s. is not a singleton. Then for any two such barriers $\tilde \R_1, \tilde \R_2$ their union is also such a barrier, as shown by  \cite{Loynes}. It follows that there exists a minimal such barrier with respect to the inclusion which then necessarily has to be $\xi$-regular.
\end{proof}
It follows that, without loss of generality, we may restrict our attention to $\xi$-regular barriers. Henceforth,
whenever a barrier is given it is assumed that it is a $\xi$-regular barrier, where the measure $\xi$ will be clear from the context.

\subsection{Root's solution and its PDE characterisation}

The main result of \cite{Root} is the following.

\begin{theorem} [\cite{Root}]\label{thm:Root}
Let $\mu_0=\delta_0$, $\eta(x) \equiv 1$, and $\mu_1$ be a centred probability measure on $\re$ with a finite second moment. Then there exists a barrier $\R^*$ such that $\sigma_{\R^*}$ is a solution of {\rm SEP}$(\bmu_1)$.
\end{theorem}

The first significant generalisation of this result is due to \cite{Rost}
who showed that the result generalised to transient Markov processes under
certain conditions. The condition that the probability measure $\mu_1$ has
finite second moment has only very recently been further relaxed to the more
natural condition that the measure has a finite first moment. This was first achieved by  \cite{GassiatOberhauserdosReis}, who have extended Root's result to the case of one-dimensional (time-inhomogeneous) diffusions using PDE methods. The result was also obtained by \cite{BeiglbockCoxHuesmann} using methods from Optimal Transport theory.

\begin{remark}\label{rk:unique}
  \cite{Loynes} showed, as used above in Lemma \ref{lem:unique}, that in Theorem \ref{thm:Root} the barrier can be taken to be regular and is then unique. 
\end{remark}

We next recall the recent work of \cite{CoxWang:11} and \cite{GassiatOberhauserdosReis}. For a function $u:(t,x)\in\re_+\times\re\longmapsto u(t,x)\in\re$, we denote by $\partial_t u$ the $t-$derivative, $Du,D^2u$ the first and second spacial derivatives, i.e. with respect to the $x$-variable, and we introduce the (heat) second order operator
 \begin{eqnarray}
 \Lc u 
 &:=&
 -\partial_tu + \frac{1}{2}\eta^2D^2u.
 \end{eqnarray}
Consider the variational inequality or obstacle problem:
 \begin{eqnarray}\label{eq:vi}
 \min\big\{-\Lc u \;,\; u-U^{\mu_1}\big\}
 =0
 &\mbox{and}&
 u(0,\cdot)=\pmu{0}.
 \end{eqnarray}

Then, based on the existence result of \cite{Root}, \cite{CoxWang:11} proved the following result.

\begin{theorem} [Theorem~4.2, \cite{CoxWang:11}; Theorem~2, \cite{GassiatOberhauserdosReis}]\label{thm:CoxWang}
Let $\bmu_1=(\mu_0,\mu_1)$ be centred probability measures on $\re$ in convex order. Then, there is a unique solution $u^1$ of \eqref{eq:vi} which extends continuously to $[0,\infty]\times[-\infty,\infty]$, and the Root solution of the {\rm SEP}$(\bmu_1)$ is induced by the regular barrier
\begin{eqnarray*}
\R^*
&=&
\big\{(t,x)\in[0,\infty]\times[-\infty,\infty]: u^1(t,x)=U^{\mu_1}(x)\big\}.
\end{eqnarray*}
Moreover, we have the representation $u^1(t,x)=-\e\big|X_{t\land \sigma_{\R^*}}-x\big|,$ for all $t\ge 0, x\in\re$.
\end{theorem}
In \cite{CoxWang:11}, the solution to the variational inequality was determined as a solution in an appropriate Sobolev space, while \cite{GassiatOberhauserdosReis} show that the solution can be understood in the viscosity sense.

\subsection{Optimal stopping characterisation}

The objective of this paper is to provide a probabilistic version of the last result, and its generalisation to the multiple marginal problem. 
Our starting point is the classical probabilistic representation of the solution to \eqref{eq:vi} as an optimal stopping problem. Define now
 \begin{eqnarray}\label{eq:opt_stop}
 u^1(t,x)
 :=
 \sup_{\tau\in\T^t}J^1_{t,x}(\tau)
 &\mbox{with}&
 J^1_{t,x}(\tau)
 :=
 \e^x\big[U^{\mu_0}(Y_\tau) + (\pmu{1}-\pmu{0})(Y_\tau)\1_{\{\tau<t\}}
   \big],
 \end{eqnarray}
where $\T^t$ is the collection of all $(\F_t)$--stopping times $\tau\le t$.
Then, using classical results, see e.g.\ \cite{BensoussanLions:82}, when properly understood, $u_1$ in \eqref{eq:opt_stop} is a solution to \eqref{eq:vi}. Uniqueness, in an appropriate sense, of solutions to \eqref{eq:vi}, then allows to deduce that the characterisation of the Root barrier given in Theorem~\ref{thm:CoxWang} corresponds to the stopping region of the optimal stopping problem \eqref{eq:opt_stop}
 \begin{eqnarray}\label{eq:Rmu_definition}
 \R^1
 &:=&
 \big\{(t,x)\in[0,\infty]\times[-\infty,\infty]: u^1(t,x)=U^{\mu_1}(x)\big\}.
 \end{eqnarray}
The probabilistic approach we develop in this paper provides a self-contained construction of the Root solution, and does not rely on the existence result of \cite{Root} or PDE results. Indeed, these follow from the following direct characterisation which is a special case of Theorem~\ref{thm:mult_marg_main} below.

\begin{theorem}\label{thm:main}
Let $\bmu_1=(\mu_0,\mu_1)$ be centred probability measures on $\bI$ in convex order. Then, $\R^1$ defined by \eqref{eq:opt_stop}-\eqref{eq:Rmu_definition} is the regular barrier inducing the Root solution of the {\rm SEP}$(\bmu_1)$. Moreover,
\begin{eqnarray*}
u^1(t,x)
&=&
-\e^{\mu_0}\big|X_{t\land \sigma_{\R^1}}-x\big|,
~~\mbox{for all}~~t\ge 0,~x\in\re.
\end{eqnarray*}
\end{theorem}

\section{Multiple Marginal Root Solution of the SEP: main results}
\label{sect:multipleRoot}
\setcounter{equation}{0}

\subsection{Iterated optimal stopping and multiple marginal barriers}
\label{subsec:mainresult}
In order to extend the Root solution to the multiple marginals SEP$(\bmu_n)$, we now introduce the following natural generalisation of the previous optimal stopping problem. Denote
 \begin{eqnarray*}
 \dU^k(x)
 :=
 U^{\mu_k}(x)-U^{\mu_{k-1}}(x),
 &\mbox{and}&
 u^0(t,x):=U^{\mu_0}(x),\quad t\in [0,\infty], x\in \bI.
 \end{eqnarray*}
The main ingredient for our construction is the following iterated sequence of optimal stopping problems:
 \begin{eqnarray}\label{eq:udefn}
 u^k(t,x) 
 :=
 \sup_{\tau\in\T^t}
  J^k_{t,x}(\tau)
 &\mbox{where}&
 J^k_{t,x}(\tau)
 :=
 \mathbb{E}^x\Big[u^{k-1}(t-\tau,Y_{\tau}) 
                          + \dU^k(Y_{\tau}) \1_{\{\tau< t\}}
                   \Big],
 ~1\le k\le n.~~
 \end{eqnarray}
The stopping regions corresponding to the above sequence of optimal stopping problems are given by:
 \begin{eqnarray}\label{eq:Rk}
  \R^k 
  := 
  \big\{(t,x) \in \Ss:
           \du^k(t,x) = \dU^k(x)
  \big\}
  &\mbox{with}&
  \du^k:=u^k-u^{k-1},
  k=1,\ldots,n,
 \end{eqnarray}
 and the optimal stopping time which solves \eqref{eq:udefn} is the first entry to $\R^k$ by the time space process starting in $(t,x)$ and running backwards in time: $\tau^t(k):=\inf\{s\geq 0: (t-s,Y_s)\in \R^k\}\land t$. 
 
 Our main result shows that the same barriers used to stop the process running forward in time:
 \begin{eqnarray}\label{eq:sigmak}
 \sigma_{0}=0,
 ~~
 \sigma_k
 := \sigma_{\R^k}(\sigma_{k-1}) 
 = \inf \big\{ t \geq \sigma_{k-1} : (t,X_t) \in \R^k \big\},
 ~ k=1,\ldots,n,
\end{eqnarray}
give the multiple marginals Root solution of SEP$(\bmu_n)$. It is important to note that the barriers in \eqref{eq:Rk} are not necessarily nested -- both $\R^k$ and $\R^{k-1}$ may contain points which are not in the other barrier. 
\begin{figure}[th]
  \centering
  \includegraphics[width=\textwidth]{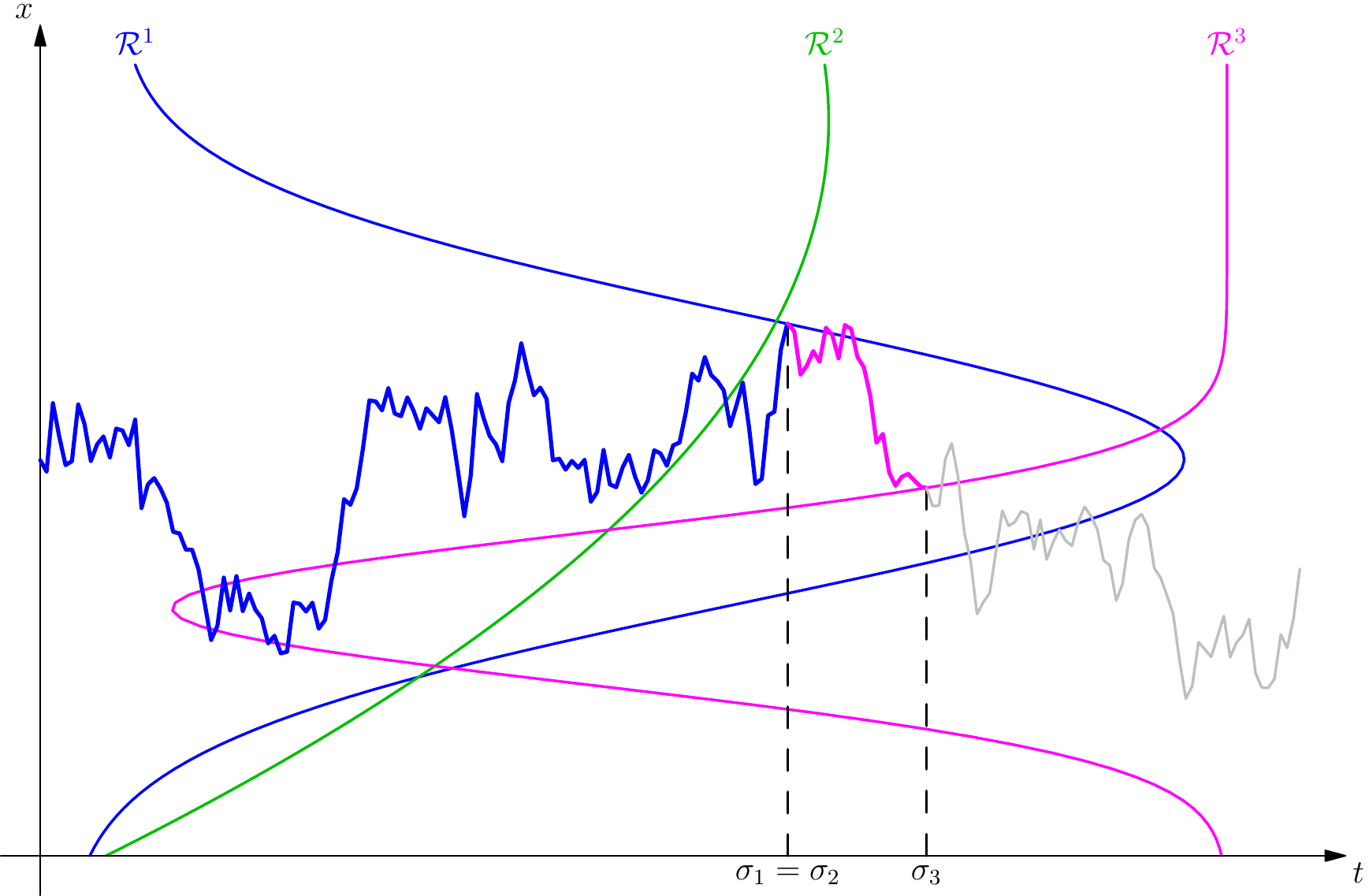}
  \caption{A realisation of a Root-type solution to the multiple marginal problem. Here we depict three barriers which are not ordered (in the sense that $\mathcal{R}^1 \supsetneq \mathcal{R}^2 \supsetneq \mathcal{R}^3$). As a result, the given realisation can enter the second and third barriers before the first stopping time. Note also that since the first stopping time, $\sigma_1$, happens at a point which is also inside the second barrier, we have here $\sigma_1 = \sigma_2$. }
  \label{fig:MultipleBarriers}
\end{figure}
An example of a possible sequence of stopping times is depicted in Figure~\ref{fig:MultipleBarriers}. Since the barriers are not necessarily nested, in general $\sigma_k$ will not be equal to the first entry time to the barrier, only the first entry time \emph{after} $\sigma_{k-1}$. It may also be the case that $\sigma_{k-1} = \sigma_k$. Both cases are shown in Figure~\ref{fig:MultipleBarriers}.

Finally, it will be useful to introduce the (time-space) measures on $\Ss$
defined for all Borel subsets $A$ of  $\Ss$ by:
\begin{eqnarray*}
  \xi^k[A] 
  :=
  \p^{\mu_0}\big[(\sigma_{k},X_{\sigma_{k}}) \in A\big],
  &k=0,\ldots,n.&
\end{eqnarray*}

We are now ready to state our main result, which includes Theorem~\ref{thm:main} as a special case.

\begin{theorem}\label{thm:mult_marg_main}
Let $\bmu_n$ be a vector of centred probability measures on $\bI$ in convex order. Then $\R^k$ is a $\xi^{k-1}$-regular barrier for all $k=1,\ldots,n$, and $(\sigma_1, \sigma_2,\dots,\sigma_n)$ solves {\rm SEP}$(\bmu_n)$. Moreover, we have
  \begin{eqnarray}\label{eq:8}
    u^k(t,x) 
    = 
    -\e^{\mu_0} \big| X_{t \wedge \sigma_k} - x\big|,
    &\mbox{for all}& 
    t \ge 0,~x \in \bI,~k=1,\ldots,n.
  \end{eqnarray}
\end{theorem}

Our proof will proceed by induction. Its main ingredients will be summarised in Section \ref{sect:induction}. 

\subsection{Optimality}
\label{sec:optimality}

In this section, we show optimality of the constructed $n$-fold Root solution of the multiple marginal Skorokhod embedding problem. 
We recall the main ingredients of our embedding defined in \eqref{eq:udefn}--\eqref{eq:sigmak}. We also denote $\ovl t^k:=\ovl t_{\R^k}$. Define the set of all solutions to {\rm SEP}$(\bmu_n)$ in \eqref{eq:nSEP}:
 \begin{eqnarray*}
 \T(\bmu_n)
 &:=&
 \big\{\rho=(\rho_1,\ldots,\rho_n)\in\T^n:~\rho_1\le\ldots \le \rho_n,~\mbox{and}~X_{\rho_i}\sim\mu_i,
                                                     ~i=1,\ldots,n 
 \big\}.
 \end{eqnarray*}
For a given function $f:\re\longrightarrow\re_+$ we consider the optimal $n$-fold embedding problem:
 \begin{eqnarray}\label{opt-prob}
 \inf_{\rho\in\T(\bmu_n)}\; \e^{\mu_0}\Big[\int_0^{\rho_n}f(t)dt\Big].
 \end{eqnarray}
 \begin{theorem}
\label{thm:optimal}
Let $\bmu_n$ be a vector of centred probability measures on $\bI$ in convex order and $f$ a non-negative non-decreasing function. Then the $n$-tuple $\sigma=(\sigma_1,\ldots,\sigma_n)$ in \eqref{eq:sigmak} is a solution of \eqref{opt-prob}:
 \begin{eqnarray*}
 \sigma\in\T(\bmu_n)
 ~~\mbox{and}~~
 \e^{\mu_0}\Big[\int_0^{\sigma_n}f(t)dt\Big]
 \;\le\;
 \e^{\mu_0}\Big[\int_0^{\rho_n}f(t)dt\Big]
 &\mbox{for all}&
 \rho\in\T(\mathbf{\bmu}_n).
 \end{eqnarray*}
\end{theorem}
The above remains true for any stopping times $\rho_1, \ldots,\rho_n$ which embed $\bmu$ since if $\rho$ is not uniformly integrable then it is not minimal, see \cite[Section 8]{Obloj:04b}, and we can find smaller stopping times $\tilde \rho\in \T(\mathbf{\bmu}_n)$ for which the above bound is already satisfied.
 
Similar to many proofs of optimality of particular solutions to {\rm SEP}, see e.g.\ \cite{Hobson:98,Cox:2008aa,HLOST}, at the heart of our argument lies identification of a suitable pathwise inequality. Interpreting \eqref{opt-prob} as an iterated Martingale Optimal Transport problem, the pathwise inequality amounts to an explicit identification of the dual optimiser in the natural Kantorovich-type duality. Our inequality is inspired by the one developed by \cite{CoxWang:11}.

For all $(t,x)\in\re_+\times\bI$ and $k=n,\ldots,0$, we introduce the functions
 \begin{align*}
   &\varphi_{n+1}(t,x) 
     :=f(t), &&
   \varphi_k(t,x)
    :=\e^{t,x}\big[\varphi_{k+1}\big(\sigma_{\R^k},X_{\sigma_{\R^k}}\big)\big],\\
   & \phi_k(x)
    :=\int_{0}^x \varphi_k(0,y)\eta(y)^{-2} dy, &&
   \psi(x)
    := 2\int_0^x \int_0^y \eta(z)^{-2} \,dz.
 \end{align*}
Our main result below involves the following functions:
 \begin{equation}
 h_k(t,x)
 :=
 \int_0^t \!\varphi_k(s,x)ds-2\int_0^x \!\phi_k(y) dy,
 ~~\mbox{and}~
 \lambda_k(x)
 :=
 (h_{k+1}-h_k)\big(\overline{t}_k(x),x\big),
 ~(t,x)\in\re_+\times\I.
 \end{equation}
\begin{lemma}\label{lem:pathwiseineq}
Let $f$ be a non-negative non-decreasing function. Then for all $(s_i,x_i)_{0\le i\le n}\subset\re_+\times\bI$, with $0 =s_0\le
  s_1 \le \dots \le s_n$, we have:
  \begin{eqnarray}
    \label{eq:pathwiseineq}
    \int_0^{s_n}f(t)dt 
    &\ge&
    \sum_{i=1}^n \lambda_i(x_i) 
    +h_1(s_0,x_0)
    + \sum_{i=1}^n \big[h_i(s_i,x_i)-h_i(s_{i-1}, x_{i-1})\big]-\psi(x_n)f(0),
  \end{eqnarray}
and equality holds if $(s_i,x_i) \in \R^i$ for $i = 1,\dots,n$.
\end{lemma}
The proof of the above inequality is entirely elementary, even if not immediate, and is reported in Appendix \ref{appendix:optimal}. 
The optimality in Theorem \ref{thm:optimal} then essentially follows  by evaluating the above on stopped paths $(\rho_i, X_{\rho_i})$ and taking expectations. Technicalities in the proof are mainly related to checking suitable integrability of various terms and the proof is also reported in Appendix \ref{appendix:optimal}. 

Finally, we note that the above pathwise inequality could be evaluated on paths of arbitrary martingale and, after taking expectations, would lead to a martingale inequality. The inequality would be sharp in the sense that we have equality for $X$ stopped at $\sigma$ in \eqref{eq:sigmak}. 
This method of arriving at martingale inequalities is linked to the so-called Burkholder method, see e.g.\ \cite{Burkholder:91}, and has been recently exploited in number of works, see e.g.\  \cite{TrajectorialDoob,BeiglbockNutz:14,OblojSpoidaTouzi}. 

\section{The inductive step}
\label{sect:induction}
\setcounter{equation}{0}

In this section we outline the main ideas behind the proof of Theorem \ref{thm:mult_marg_main}. The proof proceeds by induction. At the end of each step in the induction, we will determine a stopping time $\sx$, and the time-space distribution $\xi$, which corresponds to the distribution of the stopped process $(\sx,X_{\sx})$ under the starting measure $\mu_0$. This measure will be the key part of the subsequent definitions. Given this stopping time, and a new law $\beta$, we proceed to determine a new stopping time $\sxb$, and the corresponding time-space distribution $\xb$. This stopping time will embed the law $\beta$. This inductive step is summarised in Theorem \ref{thm:main_general_start} below. 

This stopping time $\sxb$ is constructed as the solution of an optimal stopping problem $u^\beta$, introduced below, with obstacle function appropriately defined by combining the potential function $v^\xi$ of the stopped process $X_{.\wedge \sx}$ and the difference of potentials between the starting distribution -- the spatial marginal of $\xi$ denoted $\ax$ -- and the target distribution $\beta$. We will also show that the function $u^\beta$ is equal to the potential function $v^{\xb}$, allowing us to iterate the procedure.

We now introduce the precise definitions. The measure $\mu_0$ will be a fixed integrable measure throughout, and so we will typically not emphasise the dependence of many terms on this measure. 

Let $\xi$ be the $\mathbb{P}^{\mu_0}-$time-space distribution of $(\sx,X_{\sx})$ for some UI stopping time $\sx\in\T$. The stopped potential $v^\xi$ is defined as the $\mathbb{P}^{\mu_0}-$potential of $X_{t\wedge\sx}$:
 \begin{eqnarray}\label{vxi}
 v^\xi(t,x)
 :=
 -\mathbb{E}^{\mu_0}\big[ |X_{t\wedge\sx}-x|\big],
 &t\ge 0,&
 x\in\bI.
 \end{eqnarray} 

Motivated by the iterative optimal stopping problems \eqref{eq:udefn}, we also introduce, for any probability measure $\beta$ on $\bI$, the difference of potentials
 \begin{eqnarray*}
 w^\beta:=U^\beta-U^{\ax}
 &\mbox{where}&
 \ax(A):= \xi([0,\infty)\times A), A\in \mathcal{B}(\bI),
 \end{eqnarray*}
and $\ax \preceq_{\text{\rm cx}} \beta$ is equivalent to $w^\beta\le 0$. Moreover, since $\sigma^\xi$ is UI, we have
 \begin{eqnarray}\label{vxi0}
 \mu_0 \preceq_{\text{\rm cx}} \ax, 
 &v^\xi(0,.) = \pmu{0},&
 \mbox{and}~~v^\xi(t,.) \searrow v^\xi(\infty,.):=U^\ax~\mbox{pointwise as}~t\nearrow\infty.
 \end{eqnarray}
The optimal stopping problem which will serve for our induction argument is:
  \begin{eqnarray}\label{eq:opt_stop_new}
    u^\beta(t,x) 
    &:=& 
    \sup_{\tau \in\T^t} 
    \Eps{x}{v^\xi(t-\tau,Y_\tau) + w^\beta(Y_{\tau}) \1_{\{\tau<t\}}}\quad  t \ge 0, x \in \bI.
    \end{eqnarray}
We also introduce the corresponding stopping region
  \begin{eqnarray}\label{eq:Rmu_definition_new}
    \R^\beta & := \left\{(t,x) : u^\beta(t,x) = v^\xi(t,x) + w^\beta(x)\right\} ,
  \end{eqnarray}
and we set
   \begin{equation}
    \sigma^{\xi^\beta} 
    := \inf \{t \geq \sx: (t, X_t) \in \R^\beta\},
    ~\mbox{and}~
    \xi^\beta[A] 
    := \p^\xi\big[(\sigma^{\xi^\beta},X_{\sigma_{\R^\beta}}) \in A\big]
    ~\mbox{for all}~A \in \mathcal{B}(\Ss).
  \end{equation}

\begin{theorem}\label{thm:main_general_start}
Let $\sx\in\T$ with corresponding time-space distribution $\xi$, and $\beta$ an integrable measure such that $\beta \succeq_{\text{cx}} \ax$. Then $\sxb$ is a UI stopping time embedding $\beta$ and  $u^\beta = v^{\xi^\beta}$.  Moreover, $\R^\beta$ is a $\xi^\beta$-regular barrier.
\end{theorem}

We now show that Theorem~\ref{thm:mult_marg_main} follows from Theorem \ref{thm:main_general_start}.

\begin{proof}[Proof of Theorem~\ref{thm:mult_marg_main}]
Consider the first marginal. Let $\xi=\delta_{0}\otimes\mu_0$ so that $\sx = 0$, $\ax=\mu_0$, and let $\beta = \mu_1$. Then $v^\xi(t,x) = \pmu{0}(x)$ and $u^\beta,\R^\beta$ in \eqref{eq:opt_stop_new}--\eqref{eq:Rmu_definition_new} are equal to, respectively, $u^1,\R^1$ in \eqref{eq:udefn}-\eqref{eq:Rk}. It follows from Theorem \ref{thm:main_general_start} that the stopping time $\sigma_1 =\sxb$ induced by $\R^1 = \R^\beta$ is a UI stopping time solving {\rm SEP}$(\bmu_1)$ and $u^1=u^\beta=v^{\xi^\beta}$, as required. 
We next iterate the arguments. Given the UI stopping time $\sigma_{k}$ from the $k^{\text{th}}$ step with its space-time measure $\xi$ we know that $v^\xi=u^k$ so that, with $\beta=\mu_{k+1}$, we have $u^\beta=u^{k+1}$ and $\R^\beta=\R^{k+1}$. Applying Theorem~\ref{thm:main_general_start} we get that $\sigma_{k+1}$ embeds $\mu_{k+1}$, is UI and $v^{\xi^\beta}=u^{k+1}$ as required. The proof finishes after $n$ iterations.
\end{proof}

The rest of this paper is dedicated to the proof of Theorem \ref{thm:main_general_start}. The following result isolates the main steps needed for this.
 
 \begin{lemma}\label{lem:u-beta-second-prop}
 Let $\sx\in\T$ with corresponding time-space distribution $\xi$, and $\ax \preceq_{\text{\rm cx}} \beta$. Assume further that $u^\beta(t,.)\longrightarrow U^\beta$, pointwise as $t\nearrow\infty$, and $u^\beta= v^{\xi^\beta}$. Then, $\sxb$ is a UI stopping time embedding $\beta$.
\end{lemma}

\begin{proof}
From the assumptions and the definition of $v^{\xi^\beta}$ we obtain
 $$ 
 -U^\beta(x)
 =
 -\lim_{t\to\infty} v^{\xi^\beta}(t,x)
 =\lim_{t\to\infty} \e^{\mu_0}\big[ |X_{t\land \sigma^{\xi^\beta}}-x| \big]
 \ge 
 \e^{\mu_0}\big[ |X_{\sigma^{\xi^\beta}}-x| \big]
 =
 -U^{\alpha^{\xi^\beta}}(x),
$$
where the inequality follows from Fatou's Lemma. This in particular implies that $\alpha^{\xi^\beta}$ is an integrable probability measure on $\bI$, $U^{\alpha^{\xi^\beta}}(x) >-\infty$ for all $x \in \bI$, and $U^{\xi^\beta}(x)-|x-m^{\alpha^{\xi^\beta}}|\longrightarrow 0$ as $x\to \partial \bI$, where $m^{\alpha^{\xi^\beta}}:=\int x \alpha^{\xi^\beta}(dx)$. Since also $U^{\xi^\beta}(x)-|x-m^\beta|\longrightarrow 0$ as $x\to \partial \bI$, we deduce from the above inequality  $U^\beta \le U^{\alpha^{\xi^\beta}}$ that $m^{\alpha^{\xi^\beta}}=m^\beta$, and therefore $|U^{\alpha^{\xi^\beta}}(x)-U^\beta(x)|\to 0$ as $x\to \partial \bI$. Then for $x,y \in \bI$, it follows from the dominated convergence theorem that
\begin{align*}
  U^\beta(x) - U^\beta(y) 
  & = \lim_{t\to\infty} \left[ v^{\xi^\beta}(t,x)-v^{\xi^\beta}(t,y)\right]\\
  & =\lim_{t\to\infty} \e^{\mu_0}\big[ |X_{t\land \sigma^{\xi^\beta}}-y| -|X_{t\land \sigma^{\xi^\beta}}-x|\big] \\
  & = \e^{\mu_0}\big[ |X_{\sigma^{\xi^\beta}}-y|- |X_{\sigma^{\xi^\beta}}-x|\big]\\
  & = U^{\alpha^{\xi^\beta}}(x) - U^{\alpha^{\xi^\beta}}(y).
\end{align*}
In particular, $U^\beta(x) = U^{\alpha^{\xi^\beta}}(x)+c$ for some $c \in \re$, for all $x \in \bI$, and by the above, sending $x\to \partial \bI$, we see that $c=0$. We conclude that $\alpha^{\xi^\beta}=\beta$, i.e. $X_{\sigma^{\xi^\beta}}\sim\beta$, which is the required embedding property. Moreover, it follows from the Tanaka formula together with the monotone convergence theorem that
\begin{eqnarray*}
U^\beta(x)
&=&
U^{\alpha^{\xi^\beta}}(x)
 \;=\;
 -\e^{\mu_0}\big[ |X_{\sigma^{\xi^\beta}}-x|
                     \big]
 \;=\;
 U^{\mu_0}(x) - \e^{\mu_0}\big[L^x_{\sigma^{\xi^\beta}}\big],
  ~~\mbox{for all}~~
  x\in\bI.
\end{eqnarray*}
The uniform integrability of the stopping time $\sxb$ now follows from \cite[Corollary~3.4]{Elworthy:1999aa}.
\end{proof}

The pointwise convergence of $u^\beta(t,.)$ towards $U^\beta$, as $t\to\infty$ will be stated in Lemma \ref{lem:opt-stop-prop} \ref{item:v-xi-u-beta-ineq}, while the equality $u^\beta=v^{\xi^\beta}$ is more involved and will be shown through a series of results, see Lemma \ref{lem:finalequality}. 



\begin{remark}\label{rem:ubeta-vxi}
We have $u^\beta=v^{\xi^\beta}$ if and only if $(v^\xi-u^\beta)(t,x)=\e^\xi\big[L^x_{t\wedge\sigma_{\R^\beta}}\big]$, for all $t\ge 0,x\in\bI$. Indeed, by the Tanaka formula, 
\begin{eqnarray*}
  v^{\xi^\beta}(t,x)
  &=&
  U^{\mu_0}(x)
  -\Eps{\mu_0}{L_{t \wedge \sxb}^x} 
  \;=\; 
  v^\xi(t,x) - \Eps{\mu_0}{L_{t \wedge \sxb}^x-L_{t \wedge \sx}^x}.
\end{eqnarray*}
Recalling that, under $\p^\xi$, $\srb = \inf \{t > \Tx: (t,X_t) \in \R^\beta\}$, and (under $\p^{\mu_0}$), $\sxb = \inf\{t >\sx: (t,X_t) \in \R^\beta\}$. Recall that, under $\p^\xi$, the local time is set to $L_t^x = 0$ for $t \le \Tx$, by convention. Then from the strong Markov property, we have $\Eps{\mu_0}{L_{t \wedge \sxb}^x-L_{t \wedge \sx}^x} = \Eps{(\sx,X_{\sx})}{L_{t \wedge \srb}^x} = \Eps{\xi}{L_{t \wedge \srb}^x}$, and therefore:
\begin{eqnarray}\label{vxibeta-vxi}
  v^{\xi^\beta}(t,x)
  &=&
  v^\xi(t,x) -\Eps{\xi}{L_{t \wedge \sigma_{\R^\beta}}^x},
 \end{eqnarray}
justifying the claimed equivalence.
\end{remark}

\begin{remark}
  \label{rem:regularity}
  Observe that the regularity of the barrier can now be seen as an easy
  consequence of Lemma~\ref{lem:u-beta-second-prop}. Suppose (in the setting of
  Theorem~\ref{thm:main_general_start}), we have $u^{\beta} = v^{\xi^\beta}$ and
  $u^\beta(t,.) \to U^\beta$ pointwise as $t \to \infty$. From
  \eqref{vxibeta-vxi}, \eqref{vxi0} and applying monotone convergence to $\Eps{\xi}{L_{t
      \wedge \sigma_{\R^\beta}}^x}$ as $t \to \infty$, we deduce that
  \begin{equation*}
    \Eps{\xi}{L_{\sigma_{\R^\beta}}^x} = U^{\alpha^\xi}(x) - U^\beta(x) = -w^\beta(x).
  \end{equation*}
 Now suppose that $(t,x) \not\in \R^\beta$. Then $ \Eps{\xi}{L_{\sigma_{\R^\beta}}^x}=-w^\beta(x)>(v^\xi-u^\beta)(t,x)=(v^\xi-v^{\xi^\beta})(t,x)= \Eps{\xi}{L_{t\wedge\sigma_{\R^\beta}}^x}$, by \eqref{vxibeta-vxi}. In view of Remark \ref{rem:regular lt}, this shows that $\R^\beta$ is $\xi-$regular.
\end{remark}

\section{Stopped potential and the optimal stopping problem}
\label{sect:stoppedpotential}
\setcounter{equation}{0}

\subsection{Properties of the stopped potential function}
\label{sec:prop-barr-meas}

The following lemma provides some direct properties of the stopped potential. Recall the definition $m_{\mu}(t) := \Epsa{\mu}{X_t}$.

\begin{lemma}\label{lem:vxi-immediate}
Let $\sx\in\T$ with corresponding time-space distribution $\xi$. Then, $v^\xi$ is concave and $1$-Lipschitz-continuous in $x$, and non-increasing, and $v^\xi(t,x)$ is (uniformly in $x$) $\frac{1}{2}$-H\"older continuous on $[0,T]$ for all $T>0$. In addition
 \begin{eqnarray*}
   0 &\le& 
       \pmu{0}(x)-v^{\xi}(t,x) 
       \;=\; 
       \Eps{\mu_0}{L_t^x} - \Eps{\xi}{L_t^x}
       \;\le\;
           \sqrt{2C_\eta t} \,m_{\mu_0}(t)\me^{C_\eta t},
 \end{eqnarray*}
and the following identity holds in the distribution sense:
 \begin{eqnarray*}
  \left(\Lc v^\xi\right)(t,dx)
  &=&
  -\int_0^t \eta(x)^2 \, \xi(ds,dx); 
  \quad t \ge 0,~x \in \bI,
  \end{eqnarray*}
by which we mean that, for any stopping time $\sigma\le t$, we have
\begin{align}
  \Eps{x}{v^\xi(t-\sigma,Y_\sigma) } - v^\xi(t,x) & = -\Eps{x}{\int_0^\sigma \eta(Y_s)^2 \, ds\int_0^s \delta_{Y_s}(y)\, \xi(dr, dy)} \label{eq:vxiIdent} \\ & = -\int_0^t \int_\bI q_\sigma(t-s,y) \eta(y)^2 \int_0^s \xi(dr, dy)\, ds \nonumber
\end{align}
where $q_\sigma$ is the space-time density of the process $Y_s$ (started at $t$ and running backwards in time) up to the stopping time $\sigma$.
\end{lemma}

\begin{proof} 
The definition of $v^\xi(t,x)$ in \eqref{vxi} immediately shows that $v^\xi$ is concave, $1-$Lipschitz in $x$, and non-increasing in $t$. As in Remark \ref{rem:ubeta-vxi} above, using Tanaka's formula and the strong Markov property we obtain
\begin{equation}\label{eq:v_xi_alternative}
\begin{split}
v^\xi(t,x)=&\ U^{\mu_0}(x)-\mathbb{E}^{\mu_0}[L^x_{t\wedge\sx}]= U^{\mu_0}(x)-\mathbb{E}^{\mu_0}[L^x_{t}]+\mathbb{E}^{\mu_0}[(L^x_{t}-L^x_{\sx})\1_{\{\sx\le t\}}]\\
=&\ U^{\mu_0}(x)-\mathbb{E}^{\mu_0}[L^x_{t}]+ \mathbb{E}^{\xi}[L^x_{t}]\\
=&\ U^{\mu_0}(x)-\mathbb{E}^{\mu_0}[L^x_{t}]    +\int_{[0,t]\times\bI} \Eps{y}{L_{t-s}^x}\, \xi(ds,dy).
\end{split}
\end{equation}

We now consider continuity properties of $\Eps{y}{L_t^x}$. First observe that, by the martingale property of $X_t$, we have
\begin{equation*}
  \Eps{y}{\left( X_t-x\right)^2} = \Eps{y}{\left( X_t-y\right)^2} + (x-y)^2.
\end{equation*}
Using the fact that $\eta(x)^2 \le C_\eta (1+|x|^2)$ and the martingale property of $X$, we deduce
\begin{align*}
  \Eps{y}{\left( X_t-y\right)^2}
  & \le \Eps{y}{\int_0^t \eta(X_s)^2 \, ds}\\
  & \le C_\eta\Eps{y}{\int_0^t \left(1 + |X_s|^2\right) \, ds}\\
  & \le C_\eta \left( t + 2\Eps{y}{\int_0^t \left[\left(X_s-y\right)^2 + y^2\right] ds}\right)\\
  & \le 2C_\eta \left(t(1+y^2) + \Eps{y}{\int_0^t \left( X_s -y\right)^2 ds}\right),
\end{align*}
where the first inequality follows via localisation and limiting argument using Fatou's lemma and monotone convergence. It now follows by  Gr\"onwall's lemma that
\begin{align*}
  \Eps{y}{\left( X_t-y\right)^2}
  & \le 2C_\eta (1+y^2) t \me^{2C_\eta t},
\end{align*}
from which we deduce that
\begin{equation}
  \label{eq:LocTimBound}
  \Eps{y}{L_t^x} = \Epsa{y}{X_t-x} -|x-y| \le \sqrt{\Eps{y}{\left(X_t-x\right)^2}} - |x-y| \le \sqrt{2C_\eta t}\, (1+|y|)\me^{C_\eta t}.
\end{equation}
Writing $\Eps{\mu_0}{L_{t'}^x} - \Eps{\mu_0}{L_{t}^x}= \Eps{\mu_0}{\Eps{X_t}{L_{t'-t}^x}}$ for $t<t'\le T$, we see that
\begin{align*}
  \Eps{\mu_0}{L_{t'}^x} - \Eps{\mu_0}{L_{t}^x} \le \sqrt{2C_\eta (t'-t)} \, (1+m_{\mu_0}(t)) \me^{C_\eta (t'-t)} \le \sqrt{2C_\eta (t'-t)}\, (1+m^{\mu_0}(T)) \me^{C_\eta (t'-t)}
\end{align*}
and we deduce that $v^{\xi}(x,t)$ is $\half$-H\"older continuous on $[0,T]$.
Equation \eqref{eq:LocTimBound} also provides the inequality 
\begin{equation*}
v^{\xi}(t,x) \ge \pmu{0}(x) - \Eps{\mu_0}{L_t^x} \ge \pmu{0}(x) - \sqrt{2C_\eta t}\,(1+m_{\mu_0}(t))\me^{C_\eta t}.
\end{equation*}

It remains to compute $\Lc v^\xi$. First, since $v^\xi$ is non-increasing in $t$ and concave in $x$, the partial derivatives $\partial_tv^\xi$ and $D^2v$ are well-defined as distributions on $\bI$, so $\Lc v^\xi$ makes sense in terms of measures. 

We first consider the case where $\eta$ is suitably differentiable (say smooth). Note that by a monotone convergence argument, we can restrict to the case where $Y$ remains in a compact subinterval of $\iI$ up to $\sigma$, and hence is bounded. Let $p(t,x,y)$ be the transition density for the diffusion and recall that $\Eps{y}{L^x_t}=\eta(x)^2\int_0^t p(r,y,x)dr$. It follows that for an arbitrary starting measure $\nu$, we have $\Eps{\nu}{L_t^x} = \int \nu(dy) \eta(x)^2\int_0^t p(r,y,x) d r$, and we directly compute (using Kolmogorov's Forward Equation, which holds due to the smoothness assumption on $\eta$) that 
  \begin{eqnarray*}
    \Lc \Eps{\nu}{L_t^x} 
    &=& 
    \!\!\!
    \int \nu(dy) \Big(\!-\eta(x)^2 p(t,y,x) \!+\! \eta(x)^2 \int_0^t \half D^2\left(\eta(x)^2 p(r,y,x)\right) \, d r\Big) 
    \\
    &=& 
    \!\!\!
    \int \!\nu(dy) \eta(x)^2 \Big(\!-p(t,y,x) \!+\!\int_0^t \!\partial_t p(r,y,x) \, d r\Big)\\
    &=& 
    -\eta(x)^2\int \!\nu(dy) p(0,y,x) 
    = 
    -\eta(x)^2\nu(dx) 
    = 
    \half \eta(x)^2 D^2 U^\nu(dx).
  \end{eqnarray*}
Suppose in addition that $\xi$ has a smooth density with respect to Lebesgue measure (which we also denote by $\xi$). We then compute from \eqref{eq:v_xi_alternative} and the equation above that
  \begin{eqnarray*}
    \Lc v^\xi(t,dx) 
    &=&
    \Lc\int_0^t\int_{\bI} \mathbb{E}^{s,y}\big[L^x_t\big]\xi(s,y) \, ds \, dy
    \;=\;
    \Lc\int_0^t\int_{\bI} \mathbb{E}^{y}\big[L^x_{t-s}\big]\xi(s,y) \, ds \, dy
    \\
    &=&
    -\int_0^t\int_{\bI} \eta(x)^2\delta_{\{y\}}(dx) \xi(s,y) \, ds \, dy
    \;=\;
    -\int_0^t \eta(x)^2 \xi(s,x) \, ds \, dx.
  \end{eqnarray*}
  Applying It\^o's lemma, we see that
  \begin{align}
    \Eps{x}{v^\xi(t-\sigma,Y_\sigma)} & = v^\xi(t,x) - \Eps{x}{\int_0^\sigma \int_0^s \eta(Y_s)^2 \xi(r,Y_s) \, dr \, ds} \label{eq:YIto}\\
    & = v^\xi(t,x) -\int_0^t \int_\bI q_\sigma(t-s,y) \eta(y)^2 \int_0^s \xi(dr, dy) \, ds. \nonumber
  \end{align}
  
  We now argue that our results hold for an arbitrary, locally Lipschitz function $\eta$. Keeping $\xi$ fixed as above, with a smooth density, let $\eta_n$ be a sequence of Lipschitz functions obtained from $\eta$ by mollification. Note that since we are on a compact interval, $\eta$ and hence $\eta_n$ are all bounded and from the mollification, we may assume that there exists $K$ such that $\eta, \eta_n$ are all $K$-Lipschitz; moreover $\xi$ is bounded on the corresponding compact time-space set. 

Write $Y^n$ for the solution to the SDE $dY^n_t = \eta^n(Y_t^n) \, dW_t$, and note in particular, by standard results for SDEs (e.g.~\cite[Theorem~V.4.15]{Protter:2005aa}) that $\sup_{r \in [0,t]} |Y^n_r-Y_r| \to 0$ almost surely (possibly after restricting to a subsequence), and in $\Lc^1$ as $n \to \infty$. Hence, by bounded convergence, we get convergence of the corresponding expectations on the right-hand side of \eqref{eq:YIto}, as $n\to \infty$. In addition, writing $v^{\xi,n}$ for the functions corresponding to the diffusions $Y^n$, we see from the first half of the proof that the functions $v^{\xi,n}, v^\xi$ are 1-Lipschitz in $x$, and uniformly H\"older continuous in $t$, for some common H\"older coefficient. It follows from the Arzel\`a-Ascoli theorem that $v^{\xi,n}$ converge uniformly (possibly down a subsequence) to $v^\xi$. We deduce that \eqref{eq:vxiIdent} holds for general $\eta$ and smooth $\xi$.

Finally, approximating the measure $\xi$ by smooth measures through a mollification argument, and observing that the local times for the diffusion are jointly continuous in $x$ and $t$ (by \eqref{eq:1} and the discussion preceeding this equation) we conclude that we can pass to the limit on the right-hand side of \eqref{eq:v_xi_alternative}, and hence on the left-hand side of \eqref{eq:vxiIdent}. On the other hand, when $q_\sigma$ is continuous, we can also pass to the limit on the right-hand side of \eqref{eq:vxiIdent}. Moreover we can approximate $\sigma$ by a sequence of stopping times $\sigma^n \searrow \sigma$ such that $q_{\sigma^n}$ has a continuous density, and this gives us the required result after a monotone convergence argument.  \end{proof}



For the next statement, we introduce the processes
 \begin{eqnarray}\label{Vt}
 V^t:=\big\{V^t_s:=v^{\xi}(t-s,Y_s), s \in [0,t] \big\},
 &t\in[0,\infty],&
 \end{eqnarray}
where $V^\infty$ is defined through $v^\xi(\infty,x)=U^{\ax}(x)$ as in \eqref{vxi0}, i.e.\ $V^{\infty}_s=U^{\ax}(Y_s)$.

\begin{lemma} \label{lem:vxi supermart}
Let $\sx\in\T$ with corresponding time-space distribution $\xi$. Then the processes $V^t$ and $V^{t'}-V^t$ are $\p^x$-supermartingales for all $t\le t' \le\infty$, and $x\in \bI$.
\end{lemma}

\begin{proof}
  In this proof we will want to take expectations with respect to both the $X$ and $Y$ processes at the same time; we will assume that these are defined on a product space, where the processes are independent. Then we will denote expectation with respect to the $X$ process alone by $\Epsb{\mu}{X}{A}$, etc, and the filtrations generated by the respective processes by $\Fc_s^X$ and $\Fc_s^Y$.

We first prove the supermartingale property for the process $V^t$. The case $t=\infty$ is an immediate consequence of the Jensen inequality. Next, fix $t\in[0,\infty)$, and  recall that $v^{\xi}(t,x) = -\Epsba{\mu_0}{X}{X_{t \wedge \sigma^\xi}-x}$ for $t \ge 0, x \in \bI$. Then we need to show, for $0 \le u \le s$,
\begin{equation*}
  -\Epsb{\nu_0}{Y}{\Epsb{\mu_0}{X}{|X_{(t-s)\wedge \sigma^\xi}-Y_s|}\big| \Fc_u^Y} \le - \Epsba{\mu_0}{X}{X_{(t-u)\wedge\sigma^\xi}-Y_u}.
\end{equation*}
Using Hunt's switching identity 
we have
\begin{equation*}
  \Epsba{y}{Y}{x-Y_{s-u}} = \Epsba{x}{X}{X_{s-u}-y}.
\end{equation*}
Using the Strong Markov property, and $\tilde{Y}, \tilde{\e}$ to denote independent copies of $Y$ etc., we deduce
\begin{align*}
  -\Epsb{\nu_0}{Y}{\Epsb{\mu_0}{X}{|X_{(t-s)\wedge \sigma^\xi}-Y_s|}\big| \Fc_u^Y}
  & = -\Epsb{\mu_0}{X}{\tilde{\e}^{Y_u}_{Y}\left|X_{(t-s)\wedge \sigma^\xi}-\tilde{Y}_{s-u}\right|}\\
  & = -\Epsb{\mu_0}{X}{\tilde{\e}^{X_{(t-s)\wedge \sigma^\xi}}_{X}\left|\tilde{X}_{(s-u)}-Y_u\right|}\\
  & = -\Epsb{\mu_0}{X}{\left|X_{(t-s)\wedge \sigma^\xi +(s-u)}-Y_u\right|}\\
  & \le - \Epsb{\mu_0}{X}{\left|X_{(t-u)\wedge \sigma^\xi}-Y_u\right|}
\end{align*}
where, in the final line, we used Jensen's inequality and the fact that $(t-s)\wedge \sigma^\xi +(s-u) \ge (t-u)\wedge \sigma^\xi$.

Now suppose $t' \ge t$, and consider $V^{t'}_s-V^t_s$ for $0 \le s \le t \le t'$. A similar calculation to that above shows that for $u \le s$,
\begin{align*}
  \Epsb{\nu_0}{Y}{V_s^{t'}-V_s^{t}|\Fc_u^Y}
  & = \Epsb{\mu_0}{X}{\left|X_{(t-s)\wedge \sigma^\xi +(s-u)}-Y_u\right|-\left|X_{(t'-s)\wedge \sigma^\xi +(s-u)}-Y_u\right|}
  \\
& = \Epsb{\mu_0}{X}{\left|X_{(t-u)\wedge \tilde\sigma}-Y_u\right|-\left|X_{(t'-u)\wedge \tilde\sigma}-Y_u\right|},
\end{align*}
where $\tilde\sigma:=\sigma^\xi+(s-u)$. Note that for any $r'>r$, the process $|X_{r\land u}-y|-|X_{r'\land u}-y|$ is a supermartingale for $u\geq 0$. It follows that, since $\sigma^\xi\leq \tilde \sigma$,
\begin{align*}
  \Epsb{\nu_0}{Y}{V_s^{t'}-V_s^{t}|\Fc_u^Y}
 & = \Epsb{\mu_0}{X}{\left|X_{(t-u)\wedge \tilde\sigma}-Y_u\right|-\left|X_{(t'-u)\wedge \tilde\sigma}-Y_u\right|},\\
  & \le \Epsb{\mu_0}{X}{\left|X_{(t-u)\wedge \sigma^\xi}-Y_u\right|-\left|X_{(t'-u)\wedge \sigma^\xi}-Y_u\right|}\\
  & = V_u^{t'}-V_u^t.
\end{align*}


\end{proof}

\subsection{The optimal stopping problem}

In this section we derive some useful properties of the function $u^{\beta}(t,x)$.
We first state some standard facts from the theory of optimal stopping. Introduce
 \begin{eqnarray}\label{tauxt}
   \tau^t
   :=
   \inf\{s\geq 0: (t-s,Y_s)\in{\cal R}^\beta\}\wedge t,
   &\mbox{for all}&
   t\ge 0.
 \end{eqnarray}

\begin{prop}\label{prop:existence} 
Let $\sx\in\T$ with corresponding time-space distribution $\xi$, and $\ax \preceq_{\text{\rm cx}} \beta$. Then, for all $(t,x)\in\Ss$, $\tau^t\in\T^t$ is an optimal stopping rule for the problem $u^\beta$ in \eqref{eq:opt_stop_new}:
 \begin{eqnarray}\label{eq:backwards_rep}
 u^\beta(t,x)
 = \Eps{x}{v^\xi(t-\tau^t,Y_{\tau^t}) + w^{\beta}( Y_{\tau^t})\1_{\{\tau^t<t\}}},
\end{eqnarray}
and the process $\left(u^\beta(t-s,Y_s)\right)$ is a $\p^x$-martingale for $s\in [0,\tau^t]$ and 
a $\p^x$-supermartingale for $s\in [0,t]$.
\end{prop}

\begin{proof}
Recall that under $\p^{t,x}$ the diffusion $Y_r$, $r\geq t$ departs from $x$ at time $t$, and when $t=0$, we write $\p^{0,x} = \p^x$. Then we have for $0\leq s\leq t$:
\begin{equation}
 u^\beta(t-s,x)
 =
 u^t(s,x)
 :=
 \sup_{s\leq \tau\leq t}
  \e^{s,x}\big[v^\xi(t-\tau,Y_{\tau})+w^\beta(Y_\tau)\1_{\{\tau<t\}}\big].
\end{equation}
Notice that $u^t(s,x)$ is a classical optimal stopping problem with maturity $t$, and obstacle $Z_s:=v^\xi(t-s,Y_s)+w^\beta(Y_s)\1_{\{s<t\}}$, $s\in[0,t]$, satisfying the condition of upper semicontinuity under expectation, i.e. $\limsup_{n\to\infty}\e^x[Z_{\theta_n}]\le\e^x[Z_\theta]$ for any monotone sequence of stopping times $\theta_n$ converging to $\theta$. Under this condition, it is proved in \cite{ElKaroui} that the standard theory of optimal stopping holds true. In particular, the process $\left(u^\beta(t-s,Y_s)\right)_{s \le t}$ satisfies the announced martingale and supermartingale properties, and an optimal stopping time for the problem $u^t(0,x)=u^\beta(t,x)$ is
 $$
  t\wedge\inf\big\{s\geq 0: u^t(s,Y_s)=v^\xi(t-s,Y_s) + w^\beta (Y_s)\big\},
 $$
which is exactly $\tau^t$.
\end{proof}

\begin{remark}\label{rk:barriernonempty}
Note that, taking $\tau=t$ in \eqref{eq:opt_stop_new}, $u^\beta(t,x)\geq \e^x [\pmu{0}(Y_t)]=\pmu{0}(x)+\Eps{x}{\int {\pmu{0}}''(dy) L_t^y}:= U^t(x)$. Suppose $([0,t]\times \iI)\cap\R^\beta=\emptyset$ then, from \eqref{eq:backwards_rep}, $u^\beta(t,x)=U^t(x)>U^\beta(x)$ for all $x\in \re$. We now consider the cases where $\I = \re$, $\bI = (-\infty, b_{\I}]$, $b_{\I}<\infty$ and $\bI = [a_\I, b_{\I}]$ a finite interval separately. 

In the case where $\I = \re$, we have $\Eps{x}{L_t^y} \to \infty$ as $t \to \infty$, for any $x,y \in \I$. 
As $U^t(x) \to -\infty$ for all $x$ as $t \to \infty$, it is impossible that $U^t(x)>U^\beta(x)$ for all $x \in \re$ and all $t \ge 0$. So there always exists $x\in \I$ with $\ovl t^\beta(x)<\infty$ and hence $\R^\beta\neq \emptyset$.

Similarly consider the case where $\bI = (-\infty, b_{\I}]$. From the properties of the diffusion, we know that $X_t \to b_\I$ almost surely as $t \to \infty$. Moreover, since $\Eps{x}{X_t} = x$ and $-|x| \ge \pmu{0}(x) \ge -|x| -c$, for some $c \in [0,\infty)$, we must have $U^t(x) \to -b_{\I}+(x-b_\I)$ as $t \to \infty$ for $x \in \bI$. Since $\beta$ is centred, $|U^\beta(x) + |x|| \to 0$ as $x \to -\infty$, and hence we cannot have $U^t(x) > U^\beta(x)$ for all $x \in \iI$ and all $t\geq 0$. Hence there always exists $x\in \iI$ with $\ovl t^\beta(x)<\infty$ and hence $\R^\beta\neq \emptyset$.

Finally consider the case where $\bI = [a_\I, b_{\I}]$. Hence $\lim_{t \to \infty} X_t \in \{a_\I, b_\I\}$, and a similar argument to above gives $U^t(x) \to a_{\I}-\frac{x-a_\I}{b_\I-a_\I}(b_\I+a_\I)$ as $t \to \infty$, for $x \in \bI$. This limit corresponds to  $U^{\tilde \nu}(x)$, where $\tilde \nu$ is the centred measure supported on $\{a_\I, b_\I\}$, and it is easy to check that this potential is strictly smaller than the potential of any other centred measure supported on $\bI$, and so for any other measure, there always exists $x\in \iI$ with $\ovl t^\beta(x)<\infty$ and hence $\R^\beta\neq \emptyset$. The case of the measure $\tilde \nu$ is trivial, and we exclude this from subsequent arguments.
\end{remark}

\begin{lemma}
  \label{lem:opt-stop-prop}
Let $\sx\in\T$ with corresponding time-space distribution $\xi$, and $\ax \preceq_{\text{\rm cx}} \beta$. Then:
  \begin{enumerate}
  \item the function $u^\beta$ is $1$-Lipschitz-continuous in $x$, non-increasing and $u^{\beta}$ is $\frac{1}{2}$-H\"older-continuous in $t$, and there is a constant $C$ which is independent of $\beta$ such that $|u^\beta(t,x) - u^\beta(t',x)| \le C(1+|x|)\sqrt{|t-t'|}$;
    \item \label{item:u-less-b-decreasing} $u^\beta-v^\xi$ is non-increasing in $t$; in particular, $u^\beta$ is non-increasing in $t$ and concave in $x$;
  \item \label{item:v-xi-u-beta-ineq} 
  $u^\beta(0,.) = \pmu{0}$, $U^\beta \le v^\xi + w^\beta \le  u^\beta \le v^\xi$, and $u^\beta(t,.) \searrow U^\beta$ pointwise as $t \nearrow\infty$.
  \end{enumerate}
\end{lemma}

\begin{proof}
\noindent (i) The $1$-Lipschitz-continuity of $u^\beta(t,x)$ in $x$ follows directly from the Lipschitz continuity of $v^\xi$ and $w^\beta$ in $x$. Then, the $\half-$H\"older continuity in $t$ follows by standard arguments using the dynamic programming principle (for example, as a simple modification of the proof of Proposition~2.7  in \cite{Touzi:2012aa}). 
\\ 
(ii)  Let $t'>t$, fix $\eps>0$, and let $\tau'\in\T^{t'}$ be such that
  \begin{eqnarray*}
    u^\beta(t',x) - \eps 
    \le&
    \Eps{x}{v^\xi(t'-\tau',Y_{\tau'}) + w^\beta(Y_{\tau'}) \1_{\{\tau'<t'\}}}=    \Eps{x}{V^{t'}_{\tau'} + w^\beta(Y_{\tau'}) \1_{\{\tau'<t'\}}}.
  \end{eqnarray*}
 Recall from Lemma \ref{lem:vxi supermart} the supermartingale properties of the process $V^t$ introduced in \eqref{Vt}. Then
 $$
    \Eps{x}{V^{t'}_{\tau'}}
    \;\le\;
    \Eps{x}{V^{t'}_{t\wedge\tau'}}
    \;=\;
    \Eps{x}{V^{t'}_{t\wedge\tau'}-V^t_{t\wedge\tau'}}
     + \Eps{x}{V^t_{t\wedge\tau'}}
    \;\le\;
    V^{t'}_0-V^t_0 + \Eps{x}{V^t_{t\wedge\tau'}}.
  $$
 In addition, since $w^\beta\le 0$, we have:
  $$
    \Eps{x}{w^\beta(Y_{\tau'}) \1_{\{\tau' < t'\}} }
    \;\le\; 
    \Eps{x}{w^\beta(Y_{\tau'}) \1_{\{\tau' < t\}}}
    \;=\; 
    \Eps{x}{w^\beta(Y_{\tau' \wedge t}) \1_{\{\tau'<t\}}}.
  $$
Putting these together, we conclude that
  \begin{eqnarray*}
    u^\beta(t',x) - v^\xi(t',x) - \eps 
    &\le& 
    \Eps{x}{V^t_{t\wedge\tau'}
                 +w^\beta(Y_{\tau' \wedge t}) \1_{\{\tau'<t\}}} 
    - v^\xi(t,x)
    \;\le\;
    u^\beta(t,x) - v^\xi(t,x).
  \end{eqnarray*}
By the arbitrariness of $\eps>0$, this shows $u^\beta-v^\xi$ is non-increasing in $t$, and implies that $u^\beta$ inherits from $v^\xi$ the non-increase in $t$. By the supermartingale property  of the process $\left(u^\beta(t-s,Y_s)\right)_{s \in [0,t]}$ in Proposition \ref{prop:existence}, this in turns implies that $u^\beta$ is concave in $x$.
\\
(iii) By definition, $u^\beta(0,x) = v^\xi(0,x) = \pmu{0}(x)$. Since $v^\xi(t,x) \ge U^{\ax}(x)$,  we have $u^\beta(t,x) \ge v^\xi(t,x) + w^\beta(x) \ge U^\beta(x)$. On the other hand, since $w^\beta(x) \le 0$, we have $u^\beta(t,x) \le \sup_{\tau \le t} \Eps{x}{v^\xi(t-\tau,Y_\tau)}\le v^\xi(t,x)$ by the supermartingale property of $V^t$ established in the previous Lemma \ref{lem:vxi supermart}. 

In the rest of this proof, we show that $u^\beta(t,x) \to U^\beta(x)$ as $t \to \infty$ for all $x\in\bI$. We consider three cases:

- Suppose $(t_0,x)\in\R^\beta$ for some $t_0\ge 0$. Then, for any $t\geq t_0$, $\tau^t=0$ and $u^\beta(t,x)=v^\xi(t,x)+w^\beta(x)$ which converges to $U^{\ax}(x)+w^\beta(x)=U^\beta(x)$, as $t \to \infty$. 

- Suppose that $(t_n,x_n)\in\R^\beta$ for some sequence $(t_n,x_n)_{n\ge 1}$ with $x_n\to x$. Then it follows from the previous case that $u^\beta(t,x_n)\to U^\beta(x_n)$, and therefore $u^\beta(t,x)\to U^\beta(x)$ by the Lipschitz-continuity of $u^\beta$. 

- Otherwise, suppose that $[0,\infty]\times (x-\eps,x+\eps)$ does not intersect $\R^\beta$ for some $\eps>0$. Let $(a_x,b_x):=\cup(a,b)$ over all $a\le x-\eps<x+\eps\le b$ such that $[0,\infty]\times (a,b)$ does not intersect $\R^\beta$. By Remark \ref{rk:barriernonempty}, we may assume $\R^\beta$ is not empty and hence $(a_x,b_x)\neq\iI$. In the subsequent argument, we assume that $a_x$ is finite, the case where $b_x$ is finite follows by the same line of argument. The  optimal stopping time $\tau^t$ in \eqref{tauxt} satisfies  $\tau^t \geq H_{a_x, b_x} := \inf\{r \ge 0: Y_t \not \in ( a_x, b_x)\}$ and $\tau^t \to H_{a_x, b_x}$, $\p^x$-almost surely. If both $a_x$ and $b_x$ are finite, we use the inequality $u^\beta(t,x)  \ge U^\beta(x)$, together with Fatou's Lemma, Lemmas \ref{lem:vxi-immediate} and \ref{lem:vxi supermart}, and bounded convergence, to see that
  \begin{align}
  \label{Ubeta linear}
    U^\beta(x) \le \lim_{t \to \infty}u^\beta(t,x) & = \lim_{t \to \infty}\Eps{x}{v^\xi(t-\tau^t,Y_{\tau^t}) + w^\beta(Y_{\tau^t})} \nonumber\\ 
    & \leq  \lim_{t \to \infty} \Eps{x}{v^\xi(t-H_{a_x, b_x},Y_{H_{a_x, b_x}})} + \Eps{x}{w^\beta(Y_{H_{a_x, b_x}})}\\
    & \le \Eps{x}{\lim_{t \to \infty} v^\xi(t-H_{a_x, b_x},Y_{H_{a_x, b_x}}) + w^\beta(Y_{H_{a_x, b_x}})}
    = \Eps{x}{U^\beta(Y_{H_{a_x, b_x}})} \le U^\beta(x).
  \nonumber
  \end{align}
Hence  $\lim_{t \to \infty}u^\beta(t,x)=U^\beta(x)$, and $U^\beta$ is linear on $(a_x,b_x)$. 

For the general case where $b_x$ may be infinite, a more careful argument is needed. Since $w^\beta:=(U^\beta-U^{\ax})(x) \to 0$ as $|x| \to 0$, it follows that $\delta:= \max (-w^\beta)<\infty$. Fix $\eps>0$ and choose $c$ sufficiently large that $\delta/(c-a_x) <\eps$.  Let $H_c:= \inf\{s \ge 0:  Y_s\ge c\}$ and note that $\tau^t\wedge H_c \to H_{a_x, c} = \inf \{ t \ge 0: Y_t \not\in(a_x, c)\}$ as $t \to \infty$.
Then by the martingale property of $u^\beta$ on $t \le \tau^t$, and the fact that $ u^\beta\le v^\xi$, we have
\begin{align*}
  u^\beta(t,x) 
  & = \Eps{x}{u^\beta(t-\tau^t\wedge H_c,Y_{\tau^t\wedge H_c})}
  \\
  & \le \Eps{x}{\1_{\{\tau^t\le H_c\}}(v^\xi+w^\beta)(t-\tau^t\wedge H_c,Y_{\tau^t\wedge H_c})
                        +\1_{\{\tau^t> H_c\}} v^\xi(t-\tau^t\wedge H_c,Y_{\tau^t\wedge H_c})
                          }
  \\
  & \le \Eps{x}{v^\xi(t-\tau^t\wedge H_c,Y_{\tau^t \wedge H_c}) + w^\beta(Y_{\tau^t \wedge H_c})\1_{\{\tau^t\wedge H_c < t\}}} + \delta\p^x[\tau^t>H_c],
\end{align*}
where we wrote $w^\beta(t,x)=w^\beta(x)$. 
Taking limits as $t \to \infty$, and using Fatou as above, it follows from the definition of $c$ that:
 \begin{equation}\label{Ubetalinear}
  U^\beta(x) 
  \;\le\; 
  \lim_{t \to \infty}u^\beta(t,x) 
  \;\le\; \Eps{x}{U^\beta(Y_{H_{a_x, c}})} +\eps
  \;=\; \frac{x-a_x}{c-a_x}U^\beta(c)+\frac{c-x}{c-a_x}U^\beta(a_x) + \eps.
 \end{equation}
Taking $\eps\searrow 0$ and using concavity of $U^\beta$ we get that $\lim_{t \to \infty}u^\beta(t,x)=U^\beta(x)$, and $U^\beta$ is linear on $(a_x, c)$. Letting $c\to \infty$ we conclude that $U^\beta$ is linear on $(a_x,\infty)$.
\end{proof}

\subsection{Existence and basic properties of the barrier}

We denote the barrier function corresponding to the regular barrier
$\R^\beta$ defined in \eqref{eq:Rmu_definition_new} with $\ovl
t^\beta:=\ovl t_{\R^\beta}$. It will be used on many occasions in our
proofs. Recall from \eqref{eq:ellDefn} the definition of the support of a measure $\mu_k$ in terms of the measure $\mu_{k-1}$. In what follows, we write $\ell^\beta, r^\beta$ for the bounds of the support of $\beta$ in terms of the measure $\ax$.

\begin{cor}\label{cor:R-is-barrier}
Let $\sx\in\T$ with corresponding time-space distribution $\xi$, and $\ax \preceq_{\text{\rm cx}} \beta$. Then, the set $\R^\beta$ is a (closed) barrier, and moreover
\begin{enumerate}
  \item $\big([0,\infty]\times(\ell^\beta,r^\beta)^c\big)$ $\subset$ ${\cal R}^\beta$;
  \item\label{item:R-gap-implies-beta-null} ${\cal R}^\beta\cap ([0,\infty]\times(a,b))=\emptyset$ if and
only if $\beta[(a,b)]=0$ and $w^{\beta}< 0$ on $ (a,b)$;
\item \label{item:t-bar-is-zero} $\ovl t^\beta(x) = 0$ if and only if $w^\beta(x) = 0$.
    \end{enumerate}
\end{cor}

\begin{proof}
For $(t,x) \in \R^\beta$, we have $u^\beta(t,x) = v^\xi(t,x) + w^\beta(x)$ and it is then immediate from \ref{item:v-xi-u-beta-ineq} and \ref{item:u-less-b-decreasing} of Lemma~\ref{lem:opt-stop-prop} that  $u^\beta(t',x) = v^\xi(t',x) + w^\beta(x)$ and so $(t',x) \in \R^\beta$, for all $t' > t$. By the continuity of $v^\xi$ and $u^\beta$, established in Lemmas \ref{lem:vxi-immediate} and \ref{lem:opt-stop-prop}, we conclude that $\R^\beta$ is a closed barrier.
\\
(i) For $x\notin (\ell^\beta,r^\beta)$, we have $U^\ax(x)=U^\beta(x)$ and hence $w^\beta(x)=0$. It follows from Lemma \ref{lem:opt-stop-prop} \ref{item:v-xi-u-beta-ineq} that $u^\beta(t,x)=v^\xi(t,x)$ and hence $(t,x)\in \R^\beta$ for all $t\geq 0$ so that $[0,\infty]\times(\ell^\beta,r^\beta)^c \subset{\cal R}^\beta$. 
\\
(ii) In the proof of Lemma \ref{lem:opt-stop-prop} (iii), it was shown that the condition ${\cal R}^\beta\cap ([0,\infty]\times(a,b))=\emptyset$ implies that $U^\beta$ is linear on $(a,b)$, i.e. $\beta[(a,b)]=0$, see \eqref{Ubetalinear}. Moreover, the last argument in (i) above also implies that $w^\beta(x) <0$ for all $x \in (a,b)$ whenever ${\cal R}^\beta\cap ([0,\infty]\times(a,b))=\emptyset$. This provides the implication $\Longrightarrow$. 

Suppose now that $\beta[(a,b)]=0$ and $w^\beta<0$ on $(a,b)$. For fixed $x \in (a,b)$, we have:
\begin{align*}
  u^\beta(t,x) & \ge \Eps{x}{v^\xi(t-H_{a,b}\wedge t, Y_{H_{a,b}\wedge t}) + w^\beta(Y_{H_{a,b}\wedge t})\1_{\{H_{a,b}<t\}}}\\
  & > \Eps{x}{v^\xi(t-H_{a,b}\wedge t, Y_{H_{a,b}\wedge t}) + w^\beta(Y_{H_{a,b}\wedge t})}\\
  & \ge v^\xi(t,x) - U^{\ax}(x) + U^\beta(x) = v^\xi(t,x) + w^\beta(x).
\end{align*}
Here we have used the strict inequality $w^\beta(y)<0$ for all $y \in
(a,b)$ to get the second line. To get the final line, we use Lemma~\ref{lem:vxi supermart} to deduce that
$\Lc v^\xi(t,dx) = -\eta(x)^2\int_0^t \xi(ds, dx) \ge -\eta(x)^2\alpha^\xi(dx) = \Lc
U^{\ax}(dx)$, and hence that $v^\xi(t-s,Y_s) + w^\beta(Y_s)$ is a
submartingale up to $H_{a,b}\wedge t$, given that $U^\beta(x)$ is
linear on $(a,b)$.

This shows that $u^\beta(t,x) >v^\xi(t,x) + w^\beta(x)$, and hence $(t,x) \not\in \R^\beta$, for all $t \ge 0$, and $x \in (a,b)$.
\\
(iii) If $w^\beta(x) = 0$ then $u^\beta(t,x) = v^\xi(t,x)$ for all $t$,
by (iii) of Lemma~\ref{lem:opt-stop-prop}, and so $(t,x) \in \R^\beta$ for
all $t \ge 0$. Recalling that $v^\xi(0,x) = u^\beta(0,x) =
\pmu{0}(x)$, we conclude that $(0,x) \in \R^\beta$ only if $w^\beta(x)
= 0$.

\end{proof}

\begin{remark}[On ${\cal R}^\beta$ having rays for arbitrary large $|x|$]\label{rk:many_rays}
  We can now deduce from the proof of the convergence $u^\beta\searrow U^\beta$, as $t\nearrow \infty$ in Lemma~\ref{lem:opt-stop-prop} \ref{item:v-xi-u-beta-ineq}, that for any $N>0$ there exist $x\le (-N)\vee a_\I <N\wedge b_\I\le y$ such that $\ovl t^\beta(x)<\infty$ and $\ovl t^\beta(y)<\infty$. In the proof, we show that for any point $x$ such that $\ovl t^\beta(x)=\infty$ either there exists points  $a < x < b$ such that $\ovl t^\beta(a),\ovl t^\beta(b)<\infty$ or there exists an $a$ less than $x$ such that for any $c$ large enough $U^\beta$ is linear on $(a,c)$. Letting $c \to \infty$, and using the fact that $U^\beta(c) + |c| \to 0$, we conclude that $U^\beta(y) = -|y|$ for all $y \ge a$. Then $U^\beta(y)\leq U^{\ax}(y)\leq U^{\mu_0}(y)\leq -|y|=U^\beta(y)$ implies $U^\beta(y)=U^\ax(y)$. In particular, $w^\beta(x) = 0$, and by Corollary \ref{cor:R-is-barrier} we contradict the initial assumption that $x$ is not in the barrier.
\end{remark}

\begin{remark}[On the structure of the stopping region] \label{rk:single-interval}
  Let $\ax, \beta$ be integrable measures in convex order. It follows from Corollary \ref{cor:R-is-barrier} that the barrier can be divided into at most countably many (possibly infinite) non-overlapping open intervals $J_1, J_2, J_3, \dots$ such that $J_k = (a_k,b_k)$, for $a_k < b_k$, on which $\ovl t^\beta(x) > 0$ for all $x \in (a_k,b_k)$ and $\left(\left( \bigcup_{k=1}^\infty J_k\right)^\complement \times [0,\infty] \right) \subseteq\R^\beta$.

Observing that in both the embedding, and the optimal stopping perspectives, the process started from $x\in J_k$ never exits each interval $J_k$, it is sufficient to consider each interval separately, noting that in such a case, $u^\beta(t,x)  = v^\xi(t,x)$ for all $t \ge 0$, and all $x \in \left(\bigcup_{k=1}^\infty J_k\right)^\complement$. In the subsequent argument, we will assume that we are on a single such interval $J_k$, which may then be finite, semi-infinite, or equal to $\iI$.  In addition, if the measures $\ax, \beta$ are in convex order, then their restrictions to each $J_k$ are also in convex order.

\end{remark}

\begin{remark}[On ${\cal R}^\beta$ for atomic measures]\label{rk:D_for_atoms}
  Let $\ax, \beta$ be integrable measures in convex order. Bearing in mind Remark~\ref{rk:single-interval}, we suppose that $\beta$ is a probability measure on $\bI$ such that for some integer $n'\ge 1$, and some ordered scalars $x_1<\ldots<x_{n'}$, we have $\sum_{i=1}^{n'}\beta[\{x_i\}]=1$ and $\beta[\{x_i\}]>0$ for all $i=1,\ldots, {n'}$. From the representation of the optimal stopping time $\tau^t$, see Proposition~\ref{prop:existence} above, and the form of the set ${\cal R}^\beta$ implied by Corollary \ref{cor:R-is-barrier}, it follows that
 \begin{eqnarray}\label{eq:u_optstop2}
 u^\beta(t,x)
 &=&
 \sup_{\tau\in\T(x_1,\ldots,x_{n'})}
 \e^x\left[v^\xi(t-\tau,Y_\tau)
           +w^{\beta}(Y_t)\mathbf{1}_{\{\tau< t\}}
     \right],
 \end{eqnarray}
where $\T(x_1,\ldots,x_{n'})$ is the set of stopping times $\tau$ such that $\tau\leq H_{x_1,x_{n'}}$ and $Y_\tau\in \{x_1,\ldots x_{n'}\}$ a.s.
\end{remark}

\section{Locally finitely supported measures}
\label{sect:finiteatoms}
\setcounter{equation}{0}

A probability measure $\beta$ is said to be $\ax-$locally finitely supported if its support intersects any compact subset of $\supp(\ax,\beta) = \ol{\{x: U^{\ax}(x) > U^\beta(x)\}}$ at a finite number of points. The measure $\beta$ is $\ax-$finitely supported if its support intersects $\supp(\ax,\beta)$ at a finite number of points. Throughout, $\ax$ will be fixed, so we will typically only refer to (locally) finitely supported measures. Observe that an integrable, centred measure $\beta$ can only be finitely supported if $\ell^\beta$ and $r^\beta$ are both finite --- indeed, in this case a locally finitely supported measure is finitely supported if and only if $r^\beta$ and $\ell^\beta$ are both finite.

\subsection{Preparation}

We start with two preliminary results which play crucial roles in the next section where we establish the main result for finitely supported measures. 
The first result is the key behind the time-reversal methodology which underpins the main results, see Section \ref{subsec:mainresult}. Here, we give a natural proof in the case where $X=B$ is a Brownian motion, when the proof has a simple intuition\footnote{Given its importance, we have discussed this result with many colleagues. Our first proof used an explicit formula for the density $p_c$ in $\p_c^x\big[B_s\in \td y,s<H_{0,c}\big]=p_c(s,x,y)dy$, see Proposition~2.8.10 p.98 in \cite{KaratzasShreve}. The current proof uses a clever coupling trick devised by Tigran Atoyan.}. In Appendix~\ref{ap:MarkovMartingales} we give a PDE proof which works in the more general diffusion setting. 

To understand the importance of the result, it is helpful to think of the local time of $X$ and of $Y$ on the two sides of the announced equality. This result is then used to obtain the key equality $v^{\xi^\beta}=u^\beta$ in a ``box" setting where the barrier is locally composed of two rays. The case of finitely supported measures is then obtained with an inductive argument in Section \ref{sec:finitelysupported}. 
\begin{lemma}\label{lem:box}
Let $L$ be the local time of a Brownian motion $B$. For any $a<x<y<b$ and $t\geq 0$ we have
$\e^x\big[ L^y_{t\land H_{a,b}}\big]=\e^y\big[ L^x_{t\land H_{a,b}}\big].$
\end{lemma}
\begin{figure}[th]
  \centering
  \includegraphics[width=\textwidth]{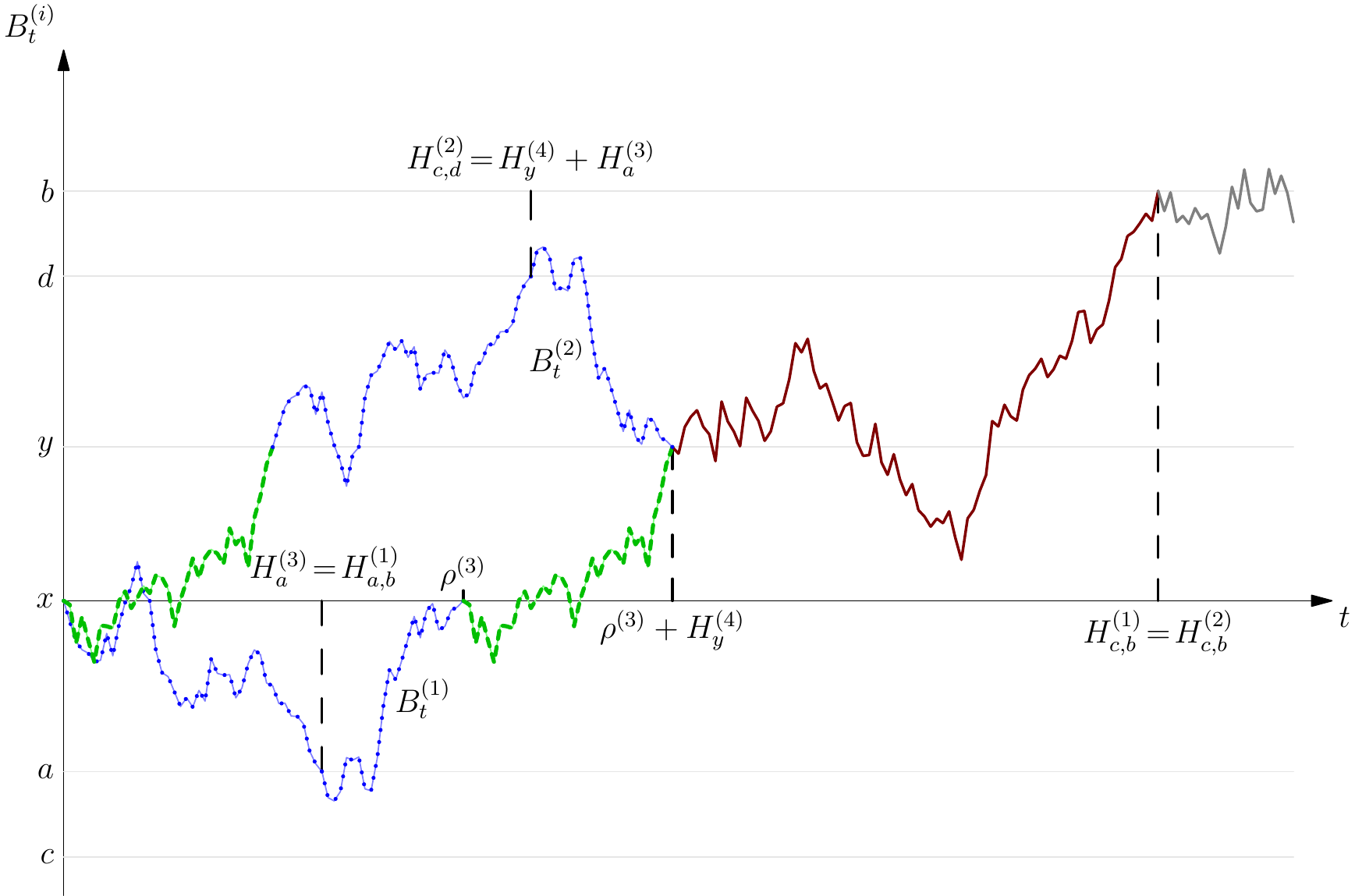}
  \caption{A depiction of the Brownian motions $B^{(1)}$ and $B^{(2)}$ constructed in the proof of Lemma~\ref{lem:box}. Observe that the blue and green sections in each process are mirror images, up to translation, while the magenta sections are equal, up to translation.}
  \label{fig:box}
\end{figure}

\begin{proof}
Without loss of generality we suppose $b-y>x-a$ and introduce two additional points $c=x-(b-y)$ and $d=y+(x-a)$ so that $c<a<x<y<d<b$ with $b-d=a-c$. Note that by translation invariance and symmetry of Brownian motion we have
$$ \e^y\big[ L^x_{t\land H_{a,b}}\big]=\e^x\big[ L^y_{t\land H_{c,d}}\big].$$
Using this in the desired equality, and subtracting $\e^x\big[ L^y_{t\land H_{c,b}}\big]$, we see that it suffices to show that
$$ \e^x\big[ L^y_{t\land H_{c,b}}- L^y_{t\land H_{a,b}}\big] = \e^x\big[ L^y_{t\land H_{c,b}}- L^y_{t\land H_{c,d}}\big].
$$
Finally, by shift invariance, we may suppose without loss of generality that $x=0$. 
Consider three independent Brownian motions $B^{(3)}, B^{(4)}, B^{(5)}$ starting from $0$ and denote $H^{(i)}$ the hitting times for $B^{(i)}$. Further, let $\rho^{(3)}=\inf\{t\geq H^{(3)}_a: B^{(3)}_t=0\}$. Define two new processes 
\begin{equation}
\begin{split}
 B^{(1)}_t &:= B^{(3)}_{t\land \rho^{(3)} } + B^{(4)}_{(t-t\land \rho^{(3)})\land H^{(4)}_y} + B^{(5)}_{t-t\land (\rho^{(3)}+H^{(4)}_y)}\\
 B^{(2)}_t & := B^{(4)}_{t\land H^{(4)}_y}-B^{(3)}_{(t-t\land H^{(4)}_y)\land \rho^{(3)}} + B^{(5)}_{t-t\land (\rho^{(3)}+H^{(4)}_y)}
\end{split}
\end{equation}
and observe these are standard Brownian motions. This construction is depicted in Figure~\ref{fig:box}. We denote $L^{y,(i)}$ the local time of $B^{(i)}$ at level $y$.

Recall that $c<a<d<b$ and consider $L^{y,{(1)}}_{t\land H^{(1)}_{c,b}}- L^{y,{(1)}}_{t\land H^{(1)}_{a,b}}$. For this quantity to be non-zero the following have to happen prior to $t$: first $B^{(1)}$ has to hit $a$ without reaching $b$, then it has to come back to $x=0$ and continue to $y$ without ever reaching $c$. This happens at time $\rho^{(3)}+H^{(4)}_y$ and from then onwards  the local time $L^{y,{(1)}}$ is counted before time $t\land H^{(1)}_{c,b}$ and we see that it simply corresponds to $L^{0,(5)}$. With a similar reasoning for $L^{y,{(2)}}$, we see that our construction gives us the desired coupling:
$$ L^{y,{(1)}}_{t\land H^{(1)}_{c,b}}- L^{y,{(1)}}_{t\land H^{(1)}_{a,b}} = L^{0,(5)}_{(t-t\land (\rho^{(3)}+H^{(4)}_y))\land H^{(5)}_{c-y,b-y}} =  L^{y,{(2)}}_{t\land H^{(2)}_{c,b}}- L^{y,{(2)}}_{t\land H^{(2)}_{c,d}}
$$
and taking expectations gives the required result.
\end{proof}

We now prove an important consequence of the above result, which will form the basis of an induction argument.

\begin{lemma} \label{lem:new_box} 
Let $\sx\in\T$ with corresponding time-space distribution $\xi$, and $\ax \preceq_{\text{\rm cx}} \beta$. Let  $a<b$ and $t_0>0$ be such that $[t_0,\infty]\times\{a,b\}\subset\R^\beta$, $(0,\infty)\times (a,b)\cap\R^\beta  = \emptyset$, and $(v^{\xi^\beta}-u^\beta)(t_0,\cdot)=0$ on $[a,b]$. Then $v^{\xi^\beta}-u^\beta=0$ on $[t_0,\infty)\times[a,b]$.
\end{lemma}

\begin{proof}
In view of Remark \ref{rem:ubeta-vxi}, and the continuity of $v^{\xi^\beta}-u^\beta$, it is sufficient to show that 
  \begin{equation}
    \label{eq:9}
    v^\xi(t,x) - u^{\beta}(t,x) +u^{\beta}(t_0,x) - v^\xi(t_0,x) = \Eps{\xi}{L_{t\wedge \srb}^x} - \Eps{\xi}{L_{t_0\wedge \srb}^x}
    ~\mbox{for}~t \ge t_0,~x \in (a,b).
  \end{equation}
  We fix $x \in (a,b)$.  Since $[t_0,\infty]\times\{a,b\}\subset\R^\beta$, $(0,\infty)\times (a,b)\cap\R^\beta = \emptyset$, we have the decomposition
  \begin{eqnarray}
    \Eps{\xi}{L_{t\wedge \srb}^x}
    - \Eps{\xi}{L_{t_0\wedge \srb}^x}
    &=&
    \Eps{\xi}{\big(L_{t \wedge \srb}^x - L_{t_0 \wedge \srb}^x\big) \1_{\{T_{\xi} < t_0\}}} 
    \nonumber\\
    && \qquad{}+ \Eps{\xi}{\big(L_{t \wedge \srb}^x-L_{T_{\xi}\wedge t}^x\big) 
                                      \1_{\{t_0 \le  T_\xi < t, X_{T_\xi} \in (a,b)\}} }
    \nonumber\\
    &=&
    \Eps{\xi}{\big(L_{t \wedge \srb}^x - L_{t_0 \wedge \srb}^x\big) 
                   \1_{\{T_{\xi} < t_0<\srb\}}} 
    \nonumber\\
    && \qquad{}+ \Eps{\xi}{\big(L_{t \wedge H_{a,b}}^x-L_{T_{\xi}\wedge t}^x\big) 
                                          \1_{\{t_0 \le  T_\xi < t, X_{T_\xi} \in (a,b)\}} }
    \nonumber\\
    &\!\!\!\!\!\!\!\!\!\!\!\!=& \!\!\!\!\!\!\!\!\!\!\!\!
    \int_{(a,b)}\!\! \Eps{(t_0,y)}{L_{t\wedge H_{a,b}}^x} \!m(dy) 
    + \!\!\int_{[t_0,t]} \int_{(a,b)} \!\!\Eps{(s,y)}{L_{t\wedge H_{a,b}}^x} \!\xi(ds,dy),
    \label{decompo:vxi-ubeta}
  \end{eqnarray}
where we introduced the measure $m(dy):=\pps{\xi}{X_{t_0} \in dy, T_\xi < t_0 < \srb}$, and used the fact that, conditional on starting in $\{t_0\} \times (a,b)$, the stopping times $\srb$ and $H_{a,b}$ are equal (and starting on $\{t_0\} \times (a,b)^\complement$, we never hit $x$ before $\srb$). Observe that
  for $y \in (a,b)$, we have
  \begin{eqnarray}
    m(dy) + \xi(dy; s \ge t_0) 
    &=& 
    \pps{\xi}{X_{t_0} \in dy, T_\xi < t_0 < \srb} + \pps{\xi}{X_{T_\xi} \in dy, T_{\xi} \ge t_0}
    \nonumber\\
    &=& 
    \pps{\xi}{X_{(t_0\wedge \srb) \vee T_{\xi}} \in dy}
    \;=:\;
    \lambda(dy),
    \label{m(dy)}
  \end{eqnarray}
since $X_{\srb} \not\in (a,b)$ by the assumptions on $\R^\beta$. Moreover, since $\sigma^\xi$ is a UI  embedding of $\ax$, it follows from the Tanaka formula that for $y\in(a,b)$, we have
 \begin{eqnarray*}
 U^{\lambda}(y)
 &=&
 U^{\ax}(y) -\Eps{\xi}{L_{t_0\wedge \srb}^y}
 \;=\;
 U^{\ax}(y) - (v^\xi-u^\beta)(t_0,y),
 \end{eqnarray*}
where the last equality follows from the assumption that $(v^{\xi^\beta}-u^\beta)(t_0,.)=0$ on $[a,b]$ together with Remark \ref{rem:ubeta-vxi}. Since $D^2U^\lambda(dy)=\lambda(dy)$, this provides by substituting in \eqref{m(dy)} that for $y \in (a,b)$:
  $$
  m(dy) 
  =
  -\half D^2 U^{\lambda}(y)dy  - \xi(dy,s \ge t_0)
  = 
  \half D^2 \left(v^\xi- u^\beta\right)(t_0,dy) + \xi(dy,s < t_0).
 $$
Plugging this expression in \eqref{decompo:vxi-ubeta}, we get
  \begin{eqnarray*}
    \Eps{\xi}{L_{t\wedge \srb}^x}
    - \Eps{\xi}{L_{t_0\wedge \srb}^x}
    &=&
    \frac12\int_{(a,b)} \Eps{(t_0,y)}{L_{t\wedge H_{a,b}}^x} D^2(v^\xi-u^\beta)(t_0,dy) 
    \\
    &&\qquad{}+ \int_{[0,t]} \int_{(a,b)} \Eps{(s\vee t_0,y)}{L_{t\wedge H_{a,b}}^x} \xi(ds,dy).
  \end{eqnarray*}
The required result now follows from the following claims involving $\zeta := \inf \{s \ge 0: (t-s, Y_s) \not\in[0,t-t_0]\times (a,b)\}$:
  \begin{eqnarray}
  \int_{(a,b)}\int_{[0,t]} \Eps{(s\vee t_0,y)}{L_{t\wedge H_{a,b}}^x} \xi(ds,dy) 
  &\!\!\!\!=& \!\!\!\!
  v^\xi(t,x)-\Eps{x}{v^\xi(t-\zeta,Y_\zeta)},
  \label{eq:15}
  \\
   \half\int_{(a,b)} \Eps{(t_0,y)}{L_{t\wedge H_{a,b}}^x} D^2 v^\xi(t_0,dy) 
  &\!\!\!\!=& \!\!\!\!
  \Eps{x}{v^\xi(t_0,Y_\zeta)}-v^\xi(t_0,x),
  \label{eq:16}
 \\
  \label{eq:20}
  -\half \int_{(a,b)} \Eps{(t_0,y)}{L_{t\wedge H_{a,b}}^x} D^2 u^\beta(t_0,dy) 
  &\!\!\!\!=& \!\!\!\!
  u^\beta(t_0,x) -u^\beta(t,x)+ \Eps{x}{v^\xi(t-\zeta,Y_\zeta)-v^\xi(t_0,Y_\zeta)},
  ~~~
  \end{eqnarray}
which we now prove.
\\
(i) To prove \eqref{eq:15}, we use It\^o's formula (possibly after mollification) to get
  $$
    v^\xi(t,x) 
   = 
    \Eps{x}{v^\xi(t-\zeta,Y_{\zeta})} 
    + \Eps{x}{\int_0^\zeta \Lc v^\xi(t-s,Y_s) \, ds}.
 $$
Using Lemma~\ref{lem:vxi-immediate} and writing $p_\zeta(r,x,y)dy := \p^x(Y_r \in dy, r < \zeta)$, this provides:
  \begin{eqnarray*}
    v^\xi(t,x)-\Eps{x}{v^\xi(t-\zeta,Y_{\zeta})}
    &=& 
    \int_{y \in (a,b)} \int_{0}^{t-t_0} \eta(y)^2 p_{\zeta}(r,x,y) \, d r \left( -\int_0^{t-r} \xi(ds,dy)\right) \\
    &=&
     \int_{y \in (a,b)} \int_{t_0}^t \eta(y)^2 p_\zeta(t-u,x,y) \, du \left( -\int_0^{u}
      \xi(ds,dy)\right) \\
    &=& 
    \int_{y \in (a,b)} \int_0^t \int_{t_0\vee s}^t \eta(y)^2 p_{\zeta}(t-u,x,y) \, du \,
    \xi(ds,dy)\\
    &=& 
    \int_{y \in (a,b),s \in [0, t]} \Eps{(s\vee t_0,y)}{L_{t\land H_{a,b}}^x}\, \xi(ds, dy).
  \end{eqnarray*}
(ii) We next prove \eqref{eq:16}. Since $v^\xi(t_0,.)$ is concave by Lemma~\ref{lem:vxi-immediate}, it follows from the It\^o-Tanaka formula that:
  $$
    \Eps{x}{v^\xi(t_0,Y_\zeta)}-v^\xi(t_0,x) 
   = 
    \half \int_{(a,b)} \Eps{x}{L_\zeta^y}\, D^2 v^\xi(t_0,dy)  \\
   = \half \int_{(a,b)} \Eps{(t_0,y)}{L_{t \wedge H_{a,b}}^x} \, D^2 v^\xi(t_0,dy),
  $$
where the last equality follows from Lemma~\ref{lem:box} together with a coordinate shift.
\\
(iii) Finally we turn to \eqref{eq:20}. Recall that $u^{\beta}= v^\xi+w^{\beta}$ on $[t_0,\infty]\times\{a,b\}\subset \R^\beta$. Then, since $Y_{\zeta} \in\{a,b\}$ on $\{\zeta < t-t_0\}$, we have:
  \begin{eqnarray*}
    u^{\beta}(t-\zeta,Y_\zeta) 
    &=& 
    u^{\beta}(t_0 ,Y_\zeta)\1_{\{\zeta=t-t_0\}} 
    + \left(v^\xi(t-\zeta,Y_\zeta)+w^{\beta}(Y_{\zeta})\right) \1_{\{\zeta <t-t_0\}}
    \\
    &=& 
    u^{\beta}(t_0,Y_\zeta) \1_{\{\zeta=t-t_0\}} 
    +\left(v^\xi(t-\zeta,Y_\zeta)+w^{\beta}(Y_{\zeta})\right) \1_{\{\zeta<t-t_0\}} 
    \\ 
    && \qquad  \qquad\qquad{} 
    + \left(v^\xi(t-\zeta,Y_\zeta) - v^\xi(t_0,Y_\zeta)\right) \1_{\{\zeta = t-t_0\}}
    \\
    &=& 
    u^{\beta}(t_0,Y_\zeta)\1_{\{\zeta=t-t_0\}} 
    + v^\xi(t-\zeta,Y_\zeta)-v^\xi(t_0,Y_\zeta)
    \\
    && \qquad \qquad\qquad {} 
    + \left(w^{\beta}(Y_{\zeta}) 
    + v^\xi(t_0,Y_\zeta)\right)\1_{\{\zeta<t-t_0\}}
    \\
    &=&
    u^{\beta}(t_0,Y_\zeta)+v^\xi(t-\zeta,Y_\zeta) - v^\xi(t_0,Y_\zeta).
  \end{eqnarray*}
We next use the fact that $[0,\infty]\times(a,b)$ does not intersect $\R^\beta$ to compute for $x\in(a,b)$ that
 \begin{eqnarray*}
    u^\beta(t,x) 
    &=& 
    \Eps{x}{u^{\beta}(t-\zeta,Y_\zeta)}
    \\
    &=& 
    \Eps{x}{u^{\beta}(t_0,Y_\zeta)+\left(v^\xi(t-\zeta,Y_\zeta) - v^\xi(t_0,Y_\zeta)\right)}
    \\
    &=& 
   u^{\beta}(t_0,x) 
   + \half \Eps{x}{\int_{(a,b)} L^y_\zeta D^2 u^{\beta}(t_0,dy)}
    +\Eps{x}{v^\xi(t-\zeta,Y_\zeta) - v^\xi(t_0,Y_\zeta)},
  \end{eqnarray*}
by application of the It\^o-Tanaka formula, due to the concavity of the function $u^\beta(t,.)$, as established in Lemma~ \ref{lem:opt-stop-prop}. We finally conclude from Lemma~\ref{lem:box}/\ref{lem:boxMarkov} that
 \begin{eqnarray*}
    u^\beta(t,x) 
    &=& 
    u^{\beta}(t_0,x) 
   + \half \int_{(a,b)} \Eps{(t_0,y)}{L_{t \wedge H_{a,b}}^x} \, D^2 u^\beta(t_0,dy)
    +\Eps{x}{v^\xi(t-\zeta,Y_\zeta) - v^\xi(t_0,Y_\zeta)}.
  \end{eqnarray*}
\end{proof}

\subsection{The case of finitely supported measures}\label{sec:finitelysupported}

We now start the proof of Theorem \ref{thm:main_general_start} for a (relatively) finitely supported probability measure $\beta$. Recall from Lemma \ref{lem:u-beta-second-prop}  and Lemma \ref{lem:opt-stop-prop} \ref{item:v-xi-u-beta-ineq} that we need to prove that $u^\beta=v^{\xi^\beta}$. When there is no risk of confusion we write $\sib$ for $\srb$. In the sequel, we will say that $\beta$ is $\ax$-supported on $n$ points if the measure $\beta$ restricted to $(\ell^\beta,r^\beta)$ is a discrete measure, supported on $n$ points.

\begin{prop} \label{prop:FinSupp}
Let $\sx\in\T$ with corresponding time-space distribution $\xi$, and $\beta$ an $\ax$--finitely supported measure such that $\ax \preceq_{\text{\rm cx}} \beta$. Then $u^\beta=v^{\xi^\beta}$ and Theorem \ref{thm:main_general_start} holds for $\beta$.
\end{prop}

The proof proceeds by induction on the number of points in the support of $\beta|_{(\ell^\beta,r^\beta)}$. The case where $\ax=\beta$ is trivial, since it follows immediately from \ref{item:t-bar-is-zero} of Corollary~\ref{cor:R-is-barrier} that $\R^\beta = \Ss$. Hence we suppose that $\ell^\beta < r^\beta$. We start with the case where $\left.\beta\right|_{(\ell^\beta,r^\beta)}$ contains no points, and therefore all mass starting in $(\ell^\beta,r^\beta)$ under $\xi$ will be embedded at the two points $\ell^\beta, r^\beta$. 

\begin{lemma}\label{lem:2points}
Let $\sigma^\xi\in\T$ with corresponding time space distribution $\xi$, and $\ax\preceq_{\text{\rm cx}}\beta$ with $\beta((\ell^\beta,r^\beta))=0$. Then $v^{\xb} = u^{\beta}$ holds for all $(t,x) \in \Ss$.
\end{lemma} 

\begin{proof}
  Note first that the convex ordering of $\beta$ and $\ax$ implies that $\ax([\ell^\beta,r^\beta])=\beta([\ell^\beta,r^\beta])$. Moreover, as we ruled out the case $\beta = \ax$ and $U^\beta$ is linear on $(\ell^\beta,r^\beta)$, 
  we have $U^{\ax}(x) > U^\beta(x)$ for all $x \in (\ell^\beta,r^\beta)$. It then follows from Corollary~\ref{cor:R-is-barrier} that $\R^\beta=\re_+\times \left(\bI \setminus (\ell^\beta,r^\beta)\right)$ and $\sib=\inf\{t\geq 0: Y_t\notin (\ell^\beta,r^\beta)\}$ is the first hitting time of $\left(\bI \setminus (\ell^\beta,r^\beta)\right)$. The result now follows from an application of Lemma~\ref{lem:new_box}.
\end{proof}

The proof of Proposition \ref{prop:FinSupp} will be complete when we establish that the following induction step  works.

\begin{lemma}\label{lem:induction}
Let $\sx\in\T$ with time-space distribution $\xi$. Assume $v^{\xi^\beta}=u^\beta$ for any $\beta \succeq_{cx} \ax$ which is $\ax$-supported on $n$ points. Then, $v^{\xi^\beta}=u^\beta$ for any measure $\beta$ which is $\ax$-supported on $n+1$ points.
\end{lemma}

\begin{proof}
Let $\beta$ be a centred probability measure $\ax$-supported on the $n+1$ ordered points $\mathbf{x}:=\{x_1,\ldots,x_{n+1}\}$, with $\beta[\{x_i\}]>0$ for all $i=1,\ldots, n+1$. By Remark \ref{rk:D_for_atoms}, the set $\R^\beta$ is of the form
 \begin{eqnarray*}
 \R^\beta 
 = 
 \big( [0,\infty] \times \left(\bI \setminus (\ell^\beta,r^\beta)\right)\big) 
 \bigcup_{i=1}^{n+1} \big([t_i,\infty) \times \{x_i\}\big)
 &\mbox{for some}&
 t_1,\ldots,t_{n+1}> 0.
 \end{eqnarray*}
Let $j$ be such that $t_j=\max_it_i$, so that $[t_j,\infty)\times \{x_j\}$ is a horizontal ray in $\R^\beta$ starting farthest away from zero.
Define a centred probability measure $\ax$--supported on $\mathbf{x}^{(-j)}:=\mathbf{x}\setminus\{x_j\}$ by conveniently distributing the mass of $\beta$ at $x_j$ among the closest neighboring points:
$$
\beta^*
=
\beta + \beta[\{x_{j}\}] \Big(-\delta_{\{x_{j}\} }
                                          +\frac{x_{j+1}-x_j}{x_{j+1}-x_{j-1}}\delta_{\{x_{j-1}\}}
                                          +\frac{x_{j}-x_{j-1}}{x_{j+1}-x_{j-1}}\delta_{\{x_{j+1}\}}
                                  \Big).
$$
{\bf 1.} Let $\Ij=(x_{j-1},x_{j+1})$. We first prove that
 \begin{eqnarray}\label{induction_step1}
 u^\beta(t,x)
 &=&
 u^{\beta^*}(t,x),\quad
 (t,x)\in\big([0,\infty]\times \bI\setminus I_j\big)\cup\big([0,t_j]\times I_j\big).
 \end{eqnarray}
By a direct calculation, we see that $U^{\beta^*}(x)=U^\beta(x)$ for $x\notin \Ij$, and $U^{\beta^*}$ is affine and strictly smaller than $U^\beta$ on $\Ij$. Consider first $x\notin \Ij$.
Recall \eqref{eq:backwards_rep} with the optimal stopping time $\tau^t$ being the minimum of $t$ and the first entry to $\R^\beta$ for the diffusion $X$ started in $(t,x)$ and running backward in time. However since $\max\{t_{j-1},t_{j+1}\}\leq t_j$ it follows that $Y_{\tau^t}\neq x_j$ on $\tau^t<t$. In consequence, we can rewrite \eqref{eq:u_optstop2} as
$$
u^\beta(t,x)
\;=\;
\sup_{\tau\in \T(\mathbf{x})}
J_{t,x}^\beta(\tau)
\;=\;
\sup_{\tau\in \T(\mathbf{x}^{(-j)})}
J_{t,x}^\beta(\tau)
\;=\;
\sup_{\tau\in \T(\mathbf{x}^{(-j)})}
J_{t,x}^{\beta^*}(\tau)
\;=\;
u^{\beta^*}(t,x)
~\mbox{for}~t\ge 0,~x\notin \Ij.
$$
An analogous argument shows $u^\beta(t,x)=u^{\beta^*}(t,x)$ for $x\in \Ij\setminus \{x_j\}$ and $t\leq t_j$ and for $x=x_j$ and $t<t_j$. By continuity of $u^\beta$ we also have $u^\beta(x_j,t_j)=u^{\beta^*}(x_j,t_j)$.
\\
{\bf 2.} We now prove that $u^\beta=v^{\xi^\beta}$ holds for all $(t,x)$. \\
2.1. From the fact that $u^\beta(t,x)=u^{\beta^*}(t,x)$, for $x\notin \Ij$, together with $\beta^*(\Ij)=0$, it follows that $\R^{\beta}=\R^{\beta^*}\cup \big([t_j,\infty)\times \{x_j\}\big)$. Consequently, for all $t\leq t_j$ and all $s\geq 0$,
\begin{eqnarray*}
X_{t\land \sigma_{\R^{\beta^*}}}
=
X_{t\land \sigma_{\R^{\beta}}}
&\mbox{and}&
X_{s\land \sigma_{\R^{\beta^*}}}\1_{\bI\setminus\Ij}(X_{s\land \sigma_{\R^{\beta^*}}})=X_{s\land \sigma_{\R^{\beta}}}\1_{\bI \setminus \Ij}(X_{s\land \sigma_{\R^{\beta}}}),
~~\mbox{a.s.}
\end{eqnarray*}
 It follows from the induction hypothesis that $u^\beta=v^{\xi^\beta}$ holds for all $x\in\re,\ t\leq t_j$, and for all $x \not\in \Ij$. 
\\
2.2. 
It remains to consider $x\in (x_{j-1},x_{j+1})$ and $t>t_j$. For $x\in (x_{j},x_{j+1})$, we now know that $u^\beta=v^{\xi^\beta}$ holds at $t=t_j$, and $\R^{\beta}$ places no points in $[0,\infty)\times (x_j,x_{j+1})$. Then, it follows from Lemma~\ref{lem:new_box} that $u^\beta=v^{\xi^\beta}$ on $(x_{j},x_{j+1})$. The same argument applies for $x\in (x_{j-1},x_{j})$.

\end{proof}

\subsection{The case of locally finitely supported measures}

In this subsection, we consider the case of measures $\beta$ which are $\ax$--finitely supported on any compact subset of $\R$, but could have an accumulation of atoms at $-\infty$ or $\infty$. We will establish Theorem \ref{thm:main_general_start} for such $\beta$ by suitably approximating $\beta$ with a sequence of measures with $\alpha^\xi-$finite support. Recall that $\ell^\beta = \sup\{x:\ax((-\infty,y]) = \beta((-\infty,y]) \ \forall y \le x\} = \sup\{x: U^{\ax}(y) = U^\beta(y)\ \forall y \le x\}$, and similarly for $r^\beta$. The desired result has already been shown when $-\infty < \ell^\beta \le r^\beta < \infty$, see Proposition \ref{prop:FinSupp}, so we consider the case where at least one of these is infinite. For simplicity, we suppose that both are infinite (and hence $\I = \re$), the case where only one is being similar. The approximation is depicted graphically in Figure~\ref{fig:localfiniteapprox}.
\begin{figure}[th]
  \centering
  \includegraphics[width=\textwidth]{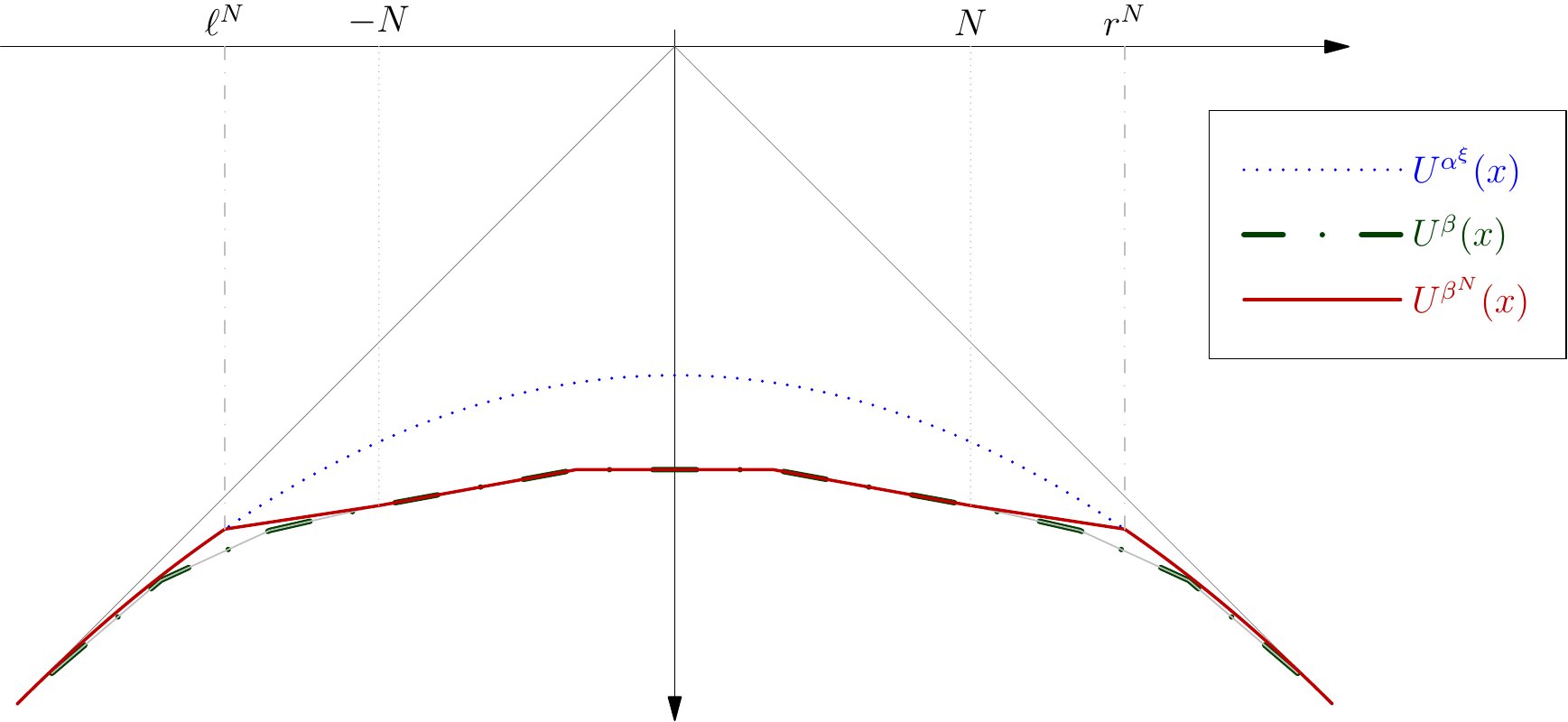}
  \caption{A graphical representation of the construction of the measure $\beta^N$ in terms of the potential functions of the measures $\alpha^\xi$ and $\beta$. }
  \label{fig:localfiniteapprox}
\end{figure}

For $N>0$, we observe that we can define a new measure $\beta^N$, and constants $\ell^N < N, r^N > N$ such that $\beta^N([-N,N]\cap A) = \beta([-N,N])\cap A)$ for $A \in \mathcal{B}(\re)$, $\beta^N([\ell^N,r^N]^\complement \cap A) = \ax([\ell^N,r^N]^\complement \cap A)$, and $\beta^N((\ell^N,-N)\cup (N,r^N)) = 0$. In particular, to construct such a measure, we can set $U^{\beta^N}(x) = U^\beta(x)$ for $x \in [-N,N]$, and extend linearly to the right of $N$, with gradient $(U^\beta)'_+(N)$ until the function meets $U^{\ax}$, at the point $r^N$, from which point on, we take $U^{\beta^N}(x) = U^{\ax}(x)$; a similar construction follows from $-N$. The existence of the point $r^N$ follows from the fact that $U^\beta(x)-U^{\alpha^\xi}(x) \to 0$ as $x \to \infty$, which in turn is a consequence of the convex ordering property. This construction guarantees
 \begin{align*}
   & U^{\beta^N}(x) \ge U^\beta(x)
   \mbox{ for all}~~x \in \re, \\ & U^{\beta^N} \mbox{ converges uniformly to } U^\beta \mbox{ and } \\
   & U^{\beta^N}(x) = U^{\ax}(x)
   ~~\mbox{for}~~x \not\in (\ell^N,r^N).
 \end{align*}
In particular, $\beta^N$ is a sequence of atomic measures with $\alpha^\xi-$finite support. Hence, by Proposition \ref{prop:FinSupp}, Theorem~\ref{thm:main_general_start} holds for these measures. Moreover, we can prove the following:

\begin{lemma}\label{prop:2}
Let $\sx\in\T$ with corresponding time-space distribution $\xi$, and $\beta$ a locally finitely supported measure such that $\ax \preceq_{\text{\rm cx}} \beta$. Let $\beta^N$ be the sequence of measures constructed above. Then the sequence $\big(\R^{\beta^N}\cap ([0,\infty)\times [-N,N])\big)_{N\ge 1}$ is non-decreasing, and 
  \begin{equation*}
    \R^\beta = \R := \overline{\bigcup_{N \ge 1}\left( \R^{\beta^N} \cap \left([0,\infty)\times [-N,N]\right) \right)}.
  \end{equation*}
\end{lemma}

\begin{proof}
  We proceed in two steps:\\
  {\bf 1.} We first show that $\big(\R^{\beta^N}\cap ([0,\infty)\times [-N,N])\big)_{N\ge 1}$ is non-decreasing and $\R^{\beta} \supseteq \R$. Recall that $U^\beta(x)\le U^{\beta^{N'}}(x) \le U^{\beta^N}(x)$ for $N'\ge N$. Then, by definition of the optimal stopping problem, we see that $u^\beta(t,x)\le u^{\beta^{N'}}(t,x) \le u^{\beta^N}(t,x)$. However, we have $U^\beta(x)=U^{\beta^{N'}}(x) = U^{\beta^N}(x)$ for $x \in [-N,N]$ by construction, and so if it is optimal to stop for $\beta^N$, it is also optimal to stop for $\beta^{N'}$ and for $\beta$. It follows that, for $x\in [-N,N]$, $(t,x)\in \R^{\beta^N}$ implies $(t,x)\in \R^{\beta^N}$ and $(t,x)\in \R^{\beta}$. The desired monotonicity follows instantly and  $\R^{\beta} \supseteq \R$ follows since $\R^\beta$ is closed.
  \\
  {\bf 2.} It remains to show the reverse inclusion $\R \supseteq \R^\beta$. \\
First, observe that for the points where $\ovl{t}_{\R}(x) = 0$ or $\ovl{t}_{\R}(x) = \infty$ the inclusion
  holds. This is an immediate consequence of Corollary~\ref{cor:R-is-barrier} together with the relation between the measures $\beta$ and $\beta^N$.\\
  The rest of the proof is devoted to showing that for a point $x$ in the support of $\beta$ with $0 < t':=\ovl t_{\R}(x) < \infty$, we have $(t,x)\notin\R^\beta$ for all $t<t'$. We first carry our preparatory computations which follow two cases. Then we combine the two to give the final result.\\
  {\bf 2.1}.\ Since Theorem~\ref{thm:main_general_start} holds for $\beta^N$, we have
  $u^{\beta^N}=v^{\xi^{\beta^N}}$. It then follows from Remark
  \ref{rem:regularity} that $\Eps{\xi}{L_{\sigma^{\beta^N}}^x} = \big(U^{\alpha^\xi}-
  U^{\beta^N}\big)(x)=:\varepsilon_0$ and $\varepsilon_0>0$ for our $x$. 
  Denote $H_{-N_0,N_0} = \inf\{t \ge \Tx: |X_t| \ge N_0\}$. Then, for
  sufficiently large $N_0$, we have $\Eps{\xi}{L_{\sigma_{\R^{\beta^N}} \wedge H_{-N_0,
        N_0}}^x} > \varepsilon_0/2$ for all $N \ge N_0$. Note that $\lim_{N\to \infty} \R^{\beta^N}\cap ([0,\infty)\times [-N_0,N_0]) = \R \cap ([0,\infty)\times [-N_0,N_0])$. 
        Letting $N \to\infty$, we conclude that
  \begin{eqnarray*}
    \Eps{\xi}{L_{\sigma_{\R} \wedge H_{-N_0, N_0}}^x} 
    &\ge& 
    \varepsilon_0/2.
  \end{eqnarray*}
  This means that, for all $t < t'$ with $t'-t$ sufficiently small there is a
  positive probability under $\p^\xi$ that the process reaches $(t,x)$ before
  hitting $\R$ (and hence also $\R^{\beta^N}$) or exiting $[-N_0,N_0]$. In
  particular, considering possible paths, we can reverse this: for any such
  $t<t'$, running backwards, there exists a positive probability that we will
  reach the support of $\xi$ before hitting $\R$ or exiting a bounded
  interval. More specifically, writing $x_- = \sup\{y < x: (0,y) \in \R\}, x_+ =
  \inf\{y>x : (0,y) \in \R\}$, and $\eps = t'-t$, for some $\eps$ sufficiently
  small at least one of the following two cases described below is true. We refer to Figure~\ref{fig:backward_cases} for a graphical interpretation of the two cases, and a
  number of the important quantities described below.
  \begin{figure}[th]
    \centering
    \includegraphics[width=\textwidth]{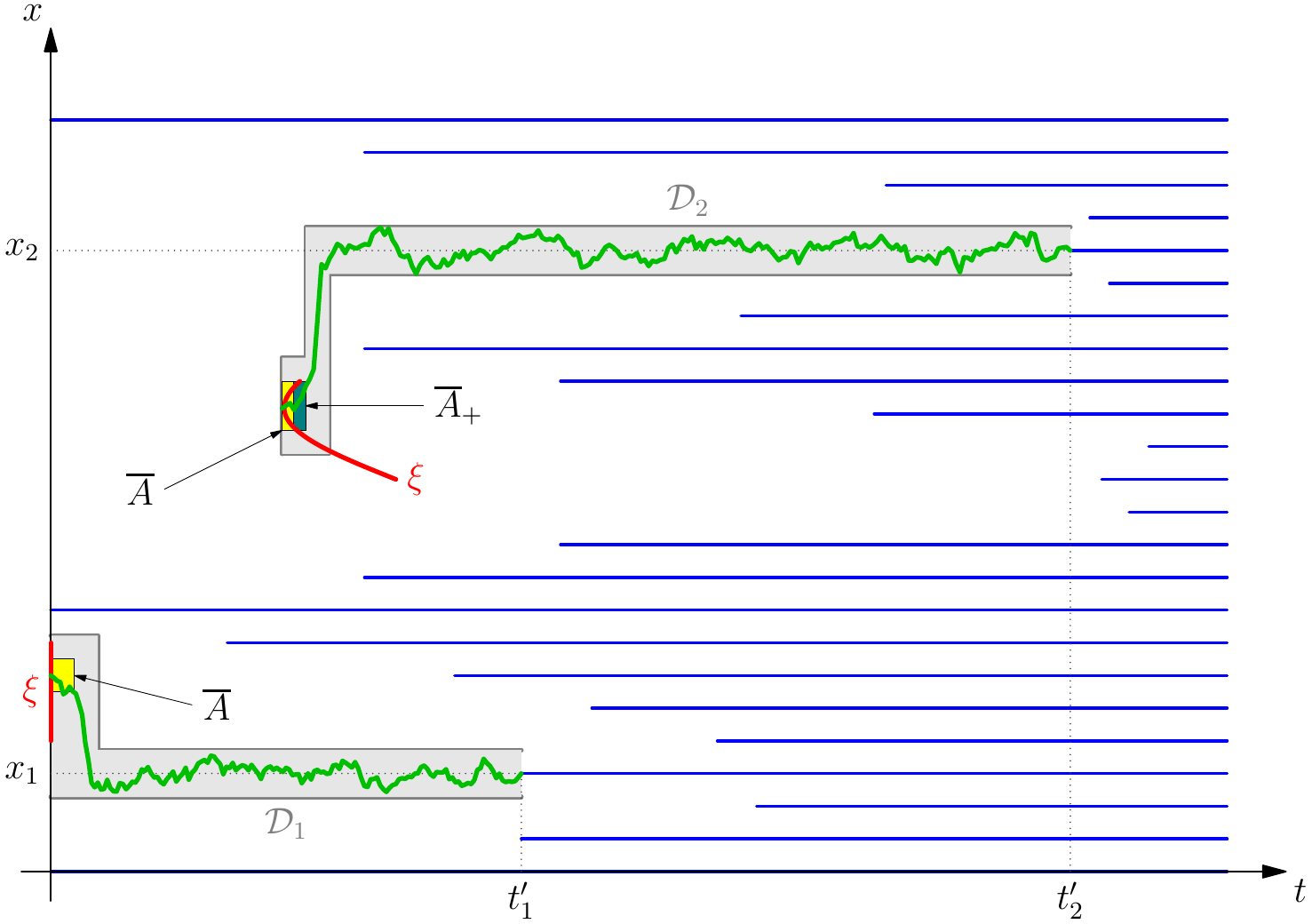}
    \caption{The possible cases considered in step 3.1.\ of the proof of
      Lemma~\ref{prop:2}. In the first case, shown in the bottom half of the
      diagram, paths starting at $(t_1',x_1)$ can only reach points in the
      support of $\xi$ (denoted by the red line) which are at time $0$. In this
      case, we are interested on the behaviour of the process on the set
      $\overline{A}$ shown, given that it does not leave the set $\D_1$. In the
      second case, the process starting at $(t_2',x_2)$ can reach points in the
      support of $\xi$ which are not in the set $\{t=0\}$. In this case, we are
      interested in the behaviour of the process on the sets $\overline{A}$ and
      $\overline{A}_+$ depicted, given that the process does not leave $\D_2$.}
    \label{fig:backward_cases}
  \end{figure}

  \begin{itemize}
  \item[Case 1] The only points of the support of $\xi$ which can be reached from
    $(t',x)$ without exiting $\R$ are in $\{0\} \times (x_-,x_+)$. Let $A
    \subseteq (x_-,x_+)$ be a closed and bounded interval such that
    $\xi(\{0\} \times A) >0$. Observe that the measures $\beta^N$ are
    $\ax$-finitely supported, and hence $\R^{\beta^N} \cap \left(\re_+ \times
      ([x-\eps,x+\eps]\setminus \{x\}) \right) = \emptyset$ for some $\eps>0$,
    and all $N$. Moreover, we may assume that $\eps$ is also sufficiently small
    that $[0,2\eps] \times [\inf A \wedge x-\eps, \sup A \vee x+\eps] \cap
    \R = \emptyset$.
    
    For such an $\eps$, write \[\D := \left( [0,2\eps] \times [\inf A
      \wedge x-\eps, \sup A \vee x+\eps] \cup [0,t') \times [x-\eps,x+\eps]\right)\] and
    note that $\R \cap \D = \emptyset$.

    Our aim is now to use the expression of $\Lc v^\xi$ in Lemma
    \ref{lem:vxi-immediate}, to show that $V^{t'}$ is a strict supermartingale
    on $\overline{A} := [0,\eps]\times A$. Recall that $t=t'-\eps$ and define
    \begin{align*}
      \tau_{N} & = \inf \{ s > 0: (t'-s,Y_s) \in \R^{\beta^N}\} \wedge t, & \tau & = \inf \{ s > 0: (t'-s,Y_s) \in \R\} \wedge t\\
      \tau_{N}^\eps & = \inf \{ s > 0: (t'-s,Y_s) \in \R^{\beta^N}\} \wedge
      t', & \tau^\eps & = \inf \{ s > 0: (t'-s,Y_s) \in \R\} \wedge t'
    \end{align*}
    and
    \begin{equation*}
      \tau^{\D} =  \inf \{ s > 0: (t'-s,Y_s) \not\in \D\}\land t'.
    \end{equation*}
    Recall the family of supermartingales $V^t$ defined in \eqref{Vt}.  We want
    to show that $\Eps{x}{V_{\tau_N}^{t'} - V_{\tau_N^{\eps}}^{t'}} \ge
    \delta >0$ for some constant $\delta$ which is independent of $N$. Since
    $\tau^{\D} \wedge t \le \tau_N\le \tau_N^\epsilon$ for all $N$, the event
    $\{\tau^{\D} >t\}$ is $\F_{\tau_N}$-measurable. Hence it is
    sufficient to show that $\Eps{x}{\left(V_{\tau_N}^{t'} -
      V_{\tau_N^{\eps}}^{t'}\right) \1_{\{\tau^{\D} >t\}}} \ge\delta$. Using
    the supermartingale property of $V^{t'}$, we can further reduce this to
    showing that \[\Eps{x}{\left(V_{\tau_N}^{t'} - V_{\tau_N^{\eps}\wedge
        \tau^{\D}}^{t'}\right) \1_{\{\tau^{\D} >t\}}} \ge \delta.\] 
        Note that on $\{\tau_D>t\}$ we have $\tau_N=t$ and $\tau_N^\epsilon\geq \tau_D$.
        We now write $q(t'-s,y)$ for the space-time density of the process
    $(t'-s,x+Y_s)$ killed when it leaves $\D$, i.e.
    \begin{equation*}
      \Eps{x}{f(Y_s); s<\tau_D} = \int q(t'-s,y) f(y) \, dy
    \end{equation*}
    for smooth functions $f$. Then from the form of $\D$, we know that $q$ is
    bounded away from zero on $\overline{A}$, and applying
    Lemma~\ref{lem:vxi-immediate} we have 
    \begin{align*}
      \Eps{x}{\left(V_{\tau_N}^{t'} - V_{\tau_N^{\eps}\wedge
          \tau^{\D}}^{t'} \right) \1_{\{\tau^{\D} >t\}}} & \ge
      -\int_{(t'-s,y) \in \overline{A}} q(t'-s,y) \mathcal{L}v^\xi(t'-s,dy) ds\\
      & \ge \int_{(t'-s,y) \in \overline{A}} \eta(y)^2 q(t'-s,y) \xi(0,dy) ds,
    \end{align*}
    by the assumption on the support of $\xi$ under consideration.
    By the assumption on $\xi$, and the fact that $q$ is bounded below on
    $\overline{A}$, this final term is strictly positive, and independent of
    $N$, so:
    \begin{equation}
      \label{eq:v-xi-eqn1}
      \Eps{x}{V_{\tau_N}^{t'} - V_{\tau_N^{\eps}}^{t'}} \ge
      \delta
    \end{equation}
    for some $\delta>0$ independent of $N$.


  \item[Case 2] There exists a bounded rectangle $\overline{A} \subset (0,t') \times
    (x_-,x_+)$ such that $\xi(\overline{A}) >0$, all points of $\overline{A}$
    can be reached from $(t',x)$ via a continuous path which does not enter
    $\R$, and the process spends a strictly positive time in
    $\overline{A}$. More specifically, for all sufficiently small $\eps>0$, we
    can choose $a_\ell, a_r$, $s_A$ such that $\overline{A} =
    [s_A,s_A+\eps/2)\times[a_\ell, a_r]$, $\xi(\overline{A}) >0$, $s_A+3\eps <
    t'$ and the set
    \begin{align*}
      \D := & \left( [s_A, s_A+\eps]\times [a_\ell-\eps, a_r+\eps]\right) \cup
      \left( [s_A+\eps,s_A+2\eps] \times [a_\ell \wedge x-\eps,a_r \vee x +
        \eps]\right) \\ & \quad \quad {}\cup \left( [s_A+2\eps,t'] \times [ x-\eps,x+\eps] \right)
    \end{align*}
    satisfies $\D \cap \R = \emptyset$. Further, recalling the definitions of
    $\tau^\D$ and $\tau_N$ above, we have $\tau^\D \le \tau_N$ $\p^x$-a.s.. In a
    similar manner to above, we now write $\tilde{q}(t'-s,y)$ for the
    space-time density of the process $(t'-s,x+Y_s)$ killed when it leaves
    $\D$, and observe that $\tilde{q}$ is bounded away from zero on the set
    $\overline{A}_+:= [s_A+\eps/2, s_A+\eps]\times[a_\ell, a_r]$. It follows
    from Lemmas~\ref{lem:vxi-immediate} and \ref{lem:vxi supermart} that:
    \begin{align*}
      & \Eps{x}{\int_0^{\tau_N} \left(\Lc v^{\xi}(t-s,Y_s) - \Lc
          v^{\xi}(t'-s,Y_s)\right) \, ds}  \\
      & \qquad \qquad \qquad \ge 
      \Eps{x}{\int_0^{\tau^D} \left(\Lc v^{\xi}(t-s,Y_s) - \Lc
          v^{\xi}(t'-s,Y_s)\right) \, ds}\\
      & \qquad \qquad \qquad \ge \int_{(t'-s,y) \in\D} \tilde{q}(t'-s,y) \left(\Lc v^{\xi}(t-s,y) - \Lc
        v^{\xi}(t'-s,y)\right) ds \, dy \\
      & \qquad \qquad \qquad \ge \int_{(t'-s,y) \in\overline{A}_+} \eta(y)^2\tilde{q}(t'-s,y)
      \xi([s_A,s_A+\eps/2),dy) ds
    \end{align*}
    where in the last line we applied Lemma~\ref{lem:vxi-immediate} and the fact
    that for $(t'-s,y) \in \overline{A}_+$
    \begin{equation*}
      \left(\Lc v^{\xi}(t-s,y) - \Lc  v^{\xi}(t'-s,y)\right) dy  =
      \eta(y)^2\xi([t-s,t'-s),dy) \ge \eta(y)^2\xi([s_A,s_A+\eps/2),dy).
    \end{equation*}

    It follows that we can choose $\delta>0$ independent of $N$ such that 
    \begin{equation*}
      \Eps{x}{\int_0^{\tau_N} \left(\Lc v^{\xi}(t-s,Y_s) - \Lc
          v^{\xi}(t'-s,Y_s)\right) \, ds} \ge \delta,
    \end{equation*}
    which, by an application of It\^o's formula, implies that
    \begin{equation}
      \label{eq:v-xi-eqn2}
      \Eps{x}{V^t_{\tau_N} - V^{t'}_{\tau_N}} \ge v^\xi(t,x) - v^\xi(t',x) + \delta. 
    \end{equation}
  \end{itemize}
  Observe finally that, in view of the supermartingale properties of Lemma \ref{lem:vxi supermart}, we can combine \eqref{eq:v-xi-eqn1} and
  \eqref{eq:v-xi-eqn2} to get:
  \begin{equation}
    \label{eq:v-xi-eqn3}
    \Eps{x}{V^t_{\tau_N} - V^{t'}_{\tau_N}} + \Eps{x}{V_{\tau_N}^{t'} - V_{\tau_N^{\eps}}^{t'}} \ge v^\xi(t,x) - v^\xi(t',x) + \delta
  \end{equation}
  for some $\delta>0$ independent of $N$, and for \emph{any} $\xi$ satisfying the
  conditions of the lemma.\\
  {\bf 2.2.}\  We are now ready to exploit the above to establish that $(t,x)\notin \R^\beta$ for $t<t'$. Take the values of $t,\eps,\delta$ determined above, and consider the
  following calculation:
  \begin{eqnarray*}
    u^{\beta^N}(t,x) - v^\xi(t,x) 
    &\ge& 
    \Eps{x}{V^t_{\tau_N}
      + w^{\beta^N}(Y_{\tau_N}) \1_{\{\tau_N<t\}} } - v^\xi(t,x)
    \\
    &\ge& 
    \Eps{x}{V^t_{\tau_N}- V^{t'}_{\tau_N}}  + \Eps{x}{V^{t'}_{\tau_N}-V^{t'}_{\tau_N^\eps}} \\
    && \quad \quad \quad {} 
       + \Eps{x}{w^{\beta^N}(Y_{\tau_N})\1_{\{\tau_N<t\}} 
       - w^{\beta^N}(Y_{\tau_N^\eps}) \1_{\{\tau_N^\eps<t'\}}}
    \\
    && \quad \quad \quad {} 
       + \Eps{x}{V^{t'}_{\tau_N^\eps} 
       + w^{\beta^N}(Y_{\tau^\eps_N})\1_{\{\tau_N^\eps<t'\}}} 
       - v^\xi(t,x)
    \\
    &\ge& 
    \left(v^\xi(t,x) - v^\xi(t',x)\right) + \delta + u^{\beta^N}(t',x) - v^\xi(t,x).
  \end{eqnarray*}
  Here we use \eqref{eq:v-xi-eqn3} for the first two terms in the second inequality;
  the third term in the second inequality is at least 0 using the fact that $\tau_N <
  t$ implies that $\tau_N^\eps=\tau_N < t$, and $w^{\beta_N}(\cdot) \le 0$.  It then
  follows, since $v^\xi$ is non-increasing in $t$, that
  \begin{equation*}
    u^{\beta^N}(t,x) - v^\xi(t,x) \ge u^{\beta^N}(t',x) - v^\xi(t',x) + \delta \ge  w^{\beta^N}(x) + \delta \ge w^{\beta}(x) + \delta.
  \end{equation*}
  We now use the fact that $\delta>0$ independently of $N$, and $u^{\beta^N}(t,x)
  \to u^{\beta}(t,x)$ as $N \to \infty$ to deduce that $u^{\beta}(t,x) -
  v^\xi(t,x) > w^{\beta}(x)$. In particular, it is not optimal to stop immediately
  for the $u^\beta$ optimal stopping problem at
  $(t,x)$ with $t<t'$, whenever $0<\ovl{t}_{\R}(x) < \infty$.\\
\end{proof}

\begin{prop} \label{prop:LocFinSupp}
Let $\sx\in\T$ with corresponding time-space distribution $\xi$, and $\beta$ a locally finitely supported measure such that $\ax \preceq_{\text{\rm cx}} \beta$. Then $u^\beta=v^{\xi^\beta}$ and Theorem \ref{thm:main_general_start} holds for $\beta$.
\end{prop}

\begin{proof}
  It follows from Lemma \ref{prop:2} that $\sigma^{\beta^N}$ decreases to
  $\sib$, and $X_{\sigma^{\beta^N}}$ converges to $X_{\sib}$ in probability, and therefore $X_{\sib} \sim \beta$. Finally, if we write $H_{\pm N} = \inf \{ t \ge \Tx : |X_t| = N\}$, we also have
\begin{align*}
  \Eps{\xi}{L_{t \wedge \sib}^x} & = \lim_{N \to \infty} \Eps{\xi}{L_{t \wedge \sib \wedge H_{\pm N}}^x} \\
  & = \lim_{N \to \infty} \Eps{\xi}{L_{t \wedge \sigma^{\beta^N} \wedge H_{\pm N}}^x} \\
  & = \lim_{N \to \infty} \left[v^\xi(t,x) -u^{\beta^N}(t,x)\right]\\
  & = v^\xi(t,x) - u^\beta(t,x),
\end{align*}
where we used \eqref{vxibeta-vxi} and monotone convergence. It follows from Remark \ref{rem:ubeta-vxi} that $v^{\xb} = u^\beta$.

Since $X_{\sib} \sim \beta$ and $X_0 \sim \alpha$, and $v^\xi(t,x) - u^\beta(t,x) \to -w^\beta(x)$ as $t \to \infty$, by monotone convergence, we have $\Eps{\xi}{L_{t \wedge \sib}^x} = -w^\beta(x)$, and hence by \cite[Corollary~3.4]{Elworthy:1999aa}, $\sib$ is a UI stopping time. Finally, we deduce that $\R^\beta$ is $\xi$-regular by observing from \eqref{eq:Rmu_definition_new} and taking limits in the  equation above that $(t,x) \in \R^\beta$ if and only if $\Eps{\xi}{L_{t \wedge \sib}^x} = w^\beta(x) = \Eps{\xi}{L_{\sib}^x}$. From Remark~\ref{rem:regular lt}, it follows that $\R^\beta$ is $\xi$-regular.
\end{proof}


\section{The general case}
\label{sect:approxatom}
\setcounter{equation}{0}

In this section, we complete the proof of Theorem \ref{thm:main_general_start}. We fix 
$\sx\in\T$ with its corresponding time-space distribution $\xi$, and let $\beta$ be an arbitrary integrable measure such that $\beta \succeq_{\text{cx}} \ax$. We start by approximating $\beta$ with a sequence of locally finitely supported measures. Let
\begin{eqnarray}\label{notations-nk}
I^k_m := [k2^{-m},(k+1)2^{-m}] \cap \bI,
&\mbox{and}&
t_m^k := \min_{x\in I_m^k} \ovl t^\beta(x)=\ovl t^\beta(x^k_m)
~\mbox{with}~(t^k_m,x^k_m)\in \R^\beta,~x^k_m\in I_m^k.
\end{eqnarray}
We set $t^k_m = \infty$ when there are no points of $\R^\beta$ in $[0,\infty) \times I^k_m$, see Corollary \ref{cor:R-is-barrier} for a characterisation. The existence of a minimizer $x^k_m$ follows from the lower semicontinuity
of the barrier function $\ovl t^\beta$ which, in turn, is implied by the
closedness property of the barrier $\R^\beta$. If there exists more than one
minimiser, we choose the smallest: $x^k_m=\min\{x\in I_m^k: \ovl
t^\beta(x)=t_m^k\}$, so that if $(t,x) = (t^k_m,x^k_m)$, then $(t,x) =
(t^{k'}_{m+1},x^{k'}_{m+1})$ for some $k'$. Note that $0\leq x^{k+1}_m-x^k_m\leq
2^{-m+1}$.

We now determine a sequence of approximating measures defined as follows: the measure $\beta^m$ is defined through its potential function, $U^{\beta^m}(x)$, and we set $U^{\beta^m}(x)$ to be the smallest concave function such that $U^{\beta^m}(x^k_m) = U^{\beta}(x^k_m)$ for all $k$. In particular, we deduce that $U^{\beta^m}(x) \le U^{\beta^{m+1}}(x) \le U^\beta(x)$; moreover, $\beta^m$ has the same mean as $\beta$, $\beta^{m} \succeq_{\text{cx}}\beta^{m+1} \succeq_{\text{cx}} \beta$ and $U^{\beta^m}(x) - U^\beta(x) \to 0$ as $x \to \partial \I$ for each $m$. This approximation is depicted in Figure~\ref{fig:GenApprox}.

\begin{figure}[ht]
  \centering
  \includegraphics[width=0.7\textwidth]{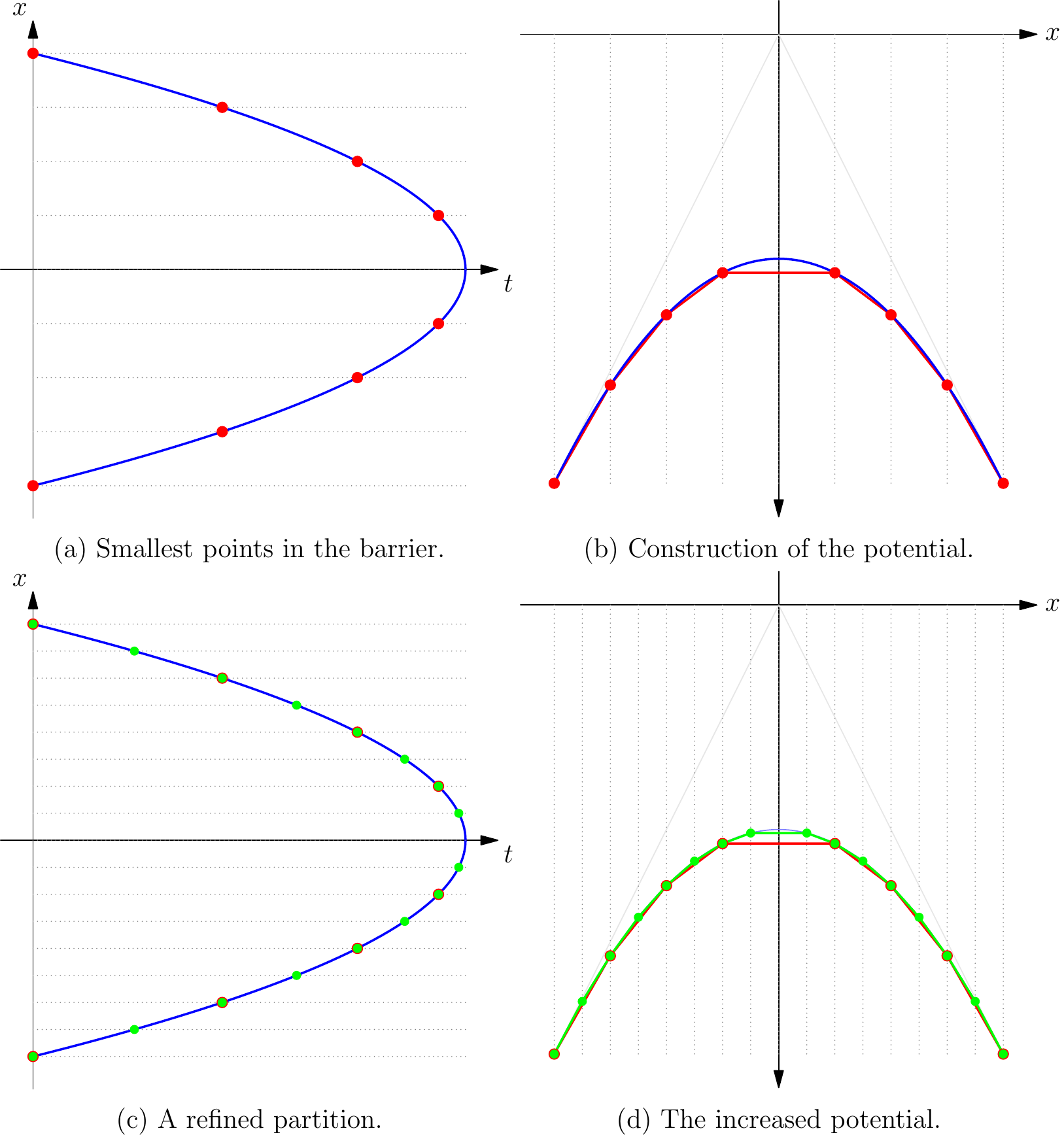}
  \caption{The approximation sequence of a general measure $\beta$. In (a), the
    red points denote the smallest point in the barrier for the given
    subdivisions (marked in gray). In (b), the original potential (in blue) is
    interpolated at the corresponding $x$-values, to produce a smaller potential
    corresponding to a measure $\beta^m$. In (c), a finer set of intervals are
    used to produce additional approximating points. Note that the previous
    (red) points are all in the new set of approximating points. In (d), these
    points are used to produce the potential of a new measure $\beta^{m+1}$.}
  \label{fig:GenApprox}
\end{figure}

Each $\beta^m$ is locally finitely supported, and so we can apply
Proposition~\ref{prop:LocFinSupp} to each $\beta^m$. Write $\R^m :=
\R^{\beta^m}$ for the corresponding barrier. 
A typical sequence of barriers are depicted in Figure~\ref{fig:BarApprox}.
\begin{figure}[ht]
  \centering
  \includegraphics[width=0.6\textwidth]{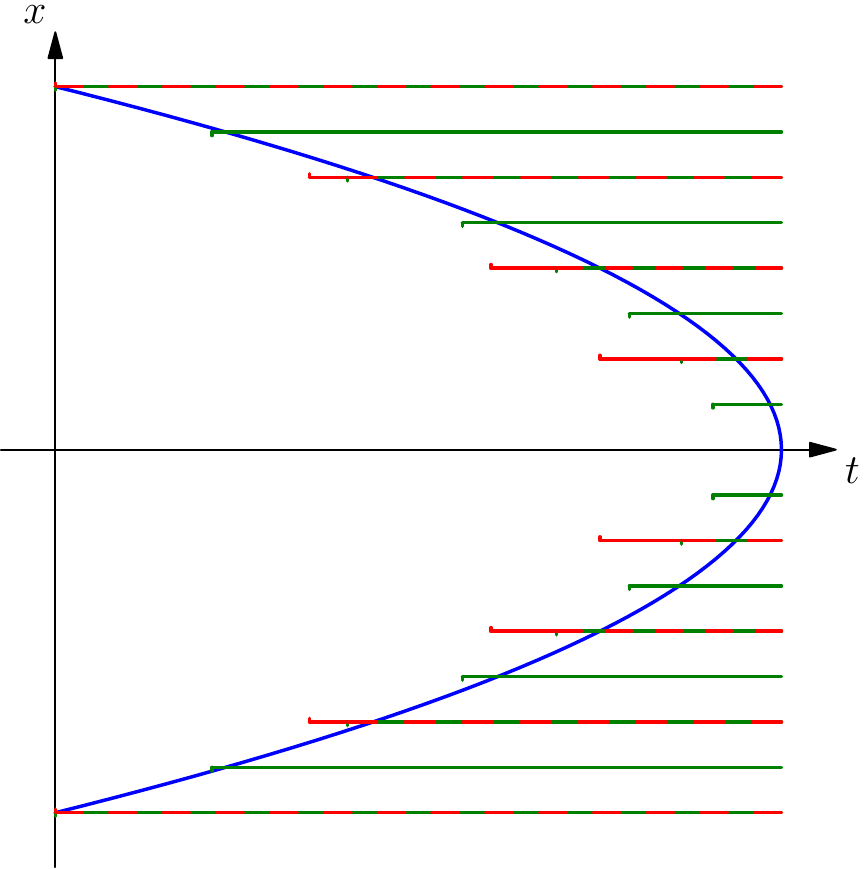}
  \caption{The sequence of barriers constructed by the approximation
    sequence. The red barrier corresponds to $\beta^m$, and the green barrier to
    $\beta^{m+1}$. Where the barriers have common atoms, the green barrier is to
    the right of the left barrier, however new `spikes' appear for the green
    barrier. The blue line denotes the barrier $\R^\beta$.}
  \label{fig:BarApprox}
\end{figure}
Since the potentials of the
measures are increasing, we have $u^{\beta^m}(t,x) \le u^{\beta^{m+1}}(t,x)$; in
addition, the function $U^{\beta^m}(x)$ is piecewise linear, and so $(t,x) \in
\R^m$ implies $x = x^k_m=x^{k'}_{m+1}$, for some $k,k'$, and $U^{\beta^m}(x) =
U^{\beta^{m+1}}(x)=U^\beta(x)$. In consequence, for such an $x$ we have
  \begin{equation}\label{eq:vuineq}
     v^\xi(t,x) + w^{\beta}(x) = v^\xi(t,x) + w^{\beta^{m+1}}(x)= v^\xi(t,x) + w^{\beta^{m}}(x)\le u^{\beta^{m}}(t,x) \le u^{\beta^{m+1}}(t,x) \le u^\beta(t,x)
  \end{equation}
and it follows from the optimal stopping formulation that $\ovl t_{\R^m}(x)\leq \ovl t_{\R^{m+1}}(x)\le \ovl t^{\beta}(x)$ --- i.e.~new spikes may
appear, but existing spikes get smaller. Taking a sequence $k_m$ such that $x =
x_m^{k_m}$ for all $m \ge m_0$, for some $m_0$, we see that $\ovl t_{\R^m}(x)$ increases to a limit. We now establish that this limit is equal to $\ovl t^\beta(x)$. 

\begin{lemma}\label{lem:1}
Let
\begin{eqnarray}
  \label{eq:Rdefn}
  &\R := 
  \ovl{\bigcap_{m \ge 0} \bigcup _{k \ge m} \R^k}.
\end{eqnarray}
Then $\R = \R^\beta$ and for any $x$ of the form $x=x_m^{k_m}$, for some sequence of indices $(k_m)$, $\ovl t_{\R^m}(x)\nearrow \ovl t^\beta(x)$. 
\end{lemma}

\begin{proof}
We first show $\R \subseteq \R^\beta$. Let $(t,x) \in \bigcap_{m \ge 0}
  \bigcup _{k \ge m} \R^k$. Then, for all $m\ge 1$, there is $k_m\ge m$ such that $(t,x)\in\R^{k_m}$, i.e. $(u^{\beta^{k_m}}-v^\xi)(t,x) = w^{\beta^{k_m}}(x) = w^\beta(x)$. However $u^{\beta^{k_m}}(t,x) \to u^\beta(t,x)$ as $m \to \infty$, and so $(u^\beta-v^\xi)(t,x)=w^\beta(x)$, proving that $(t,x)\in\R^\beta$. This shows that $\bigcap_{m \ge 0}
  \bigcup _{k \ge m} \R^k\subset \R^\beta$, and therefore $\R\subset\R^\beta$ by the closeness of $\R^\beta$.

We now show the reverse inclusion, $\R^\beta \subseteq \R$. For
  $(t,x) \in \R^\beta$, and $\eps>0$, choose $m_0$ so that $2^{-m_0}< \eps$. Then there exists $x'$ such that $|x-x'| < \eps$ and $(t',x') \in \R^{m_0}$ for some $t'$ and $\ovl t^{\beta}(x') \le
  \ovl t^\beta(x) \le t$ by our choice of points $x^k_m$. Further, as argued above, $\ovl t_{\R^{m_0}}(x')\leq \ovl t_{\R^{m}}(x') \le \ovl t^{\beta}(x')\le t$ so that $(t,x') \in \R^m$ for all $m \ge m_0$. It follows that $(t,x)\in \R$.
  
The above shows $\R=\R^\beta$, or equivalently $\ovl t_\R=\ovl t^\beta$. As observed above, for $x=x_m^{k_m}$, we have $\ovl t_{\R^m}(x)$ is an increasing sequence in $m$ and hence converges to some limit which we denote $t(x)$. By the barrier property of each $\R_m$ and the definition of $\R$ we see that $(t(x),x)\in \R$. It follows that $\ovl t_{\R}(x)\leq t(x) \leq \ovl t^\beta(x)$ and hence all three are equal.
\end{proof}

\begin{prop}
  \label{prop:3}
  Consider the approximation sequence above and define $\sigma^m = \sigma_{\R^m}\land \sigma_{\R^\beta}$. Then:
  \begin{enumerate}
  \item the process $(X_{t \wedge \sigma^m})_{t \ge \Tx}$ is uniformly
    integrable under $\p^\xi$;
  \item $\sigma^m \to \sigma_{\R^\beta}$;
  \item $\Eps{\xi}{L_{t \wedge \sigma_{\R^\beta}}^x} \le v^\xi(t,x) - u^{\beta}(t,x)$.
  \end{enumerate}
\end{prop}

\begin{proof}
(i) By definition $\sigma^m \le \sigma_{\R^m}$ and, from Proposition~\ref{prop:LocFinSupp}, the same process stopped at $\sigma_{\R^m}$ is uniformly integrable, which implies the result.
\\
(ii) 
Suppose that $\sigma^m$ does not converge a.s. to $\sigma_{\R^\beta}$. Take $\omega$ such that, possibly passing to a subsequence, we have that $\sigma^m(\omega)\to t_\infty$ for some $t_\infty< \sigma_{\R^\beta}(\omega)$. Then necessarily $(t_m,y_m):=(\sigma^m(\omega),X_{\sigma^m(\omega)})\in \R^m$ for $m$ large enough. This gives
$$     v^\xi(t_m,y_m) + w^{\beta}(y_m) = v^\xi(t_m,y_m) + w^{\beta^{m}}(y_m)= u^{\beta^{m}}(t_m,y_m).$$
We take limits on both sides. The left-hand side converges to $v^\xi(t_\infty,y_\infty)+w^\beta(y_\infty)$ by continuity of $w^\beta$ and joint continuity of $v^\xi$, see \eqref{eq:v_xi_alternative}. For the right-hand side we use $1$-Lipschitz continuity of each $u^{\beta^m}(t,\cdot)$ and $1/2$-H\"older continuity in $t$ as given in Lemma \ref{lem:opt-stop-prop}. This shows, with $X_{t_\infty}=:y_\infty$, that
$$ |u^{\beta^m}(t_\infty,y_\infty)-u^{\beta^m}(t_m,y_m)|\leq C (1+|y_\infty|) \sqrt{|t_\infty - t_m|} + |y_\infty - y_m| $$
for a constant $C$ independent of $m$, and hence 
$$\lim_{m\to \infty} u^{\beta^m}(t_m,y_m) = \lim_{m\to \infty} u^{\beta^m}(t_\infty,y_\infty) = u^\beta(t_\infty,y_\infty),$$
which then shows that $(t_\infty,y_\infty)\in \R^\beta$ and hence $\sigma_{\R^\beta}(\omega)\leq t_\infty$ which gives the desired contradiction. 
\\
(iii) Using the above, together with Proposition~\ref{prop:LocFinSupp} and Remark \ref{rem:ubeta-vxi}, we deduce that
  \begin{equation*}
    \Eps{\xi}{L_{t \wedge \sigma_{\R^\beta}}^x} =\lim_{m \to \infty}
    \Eps{\xi}{L_{t \wedge \sigma^m}^x} \le \lim_{m \to \infty}
    \Eps{\xi}{L_{t \wedge \sigma_{\R^m}}^x} = \lim_{m \to \infty} \left[
      v^\xi(t,x) - u^{\beta^m}(t,x)\right] = v^\xi(t,x) - u^\beta(t,x).
  \end{equation*}
\end{proof}

\begin{lemma}\label{lem:finalequality}
  We have $v^{\xb} = u^\beta$ and $\sxb$ is a UI stopping time embedding $\beta$.
\end{lemma}

\begin{proof} It suffices to show the first equality as the rest follows from Lemma \ref{lem:opt-stop-prop} \ref{item:v-xi-u-beta-ineq} and Lemma \ref{lem:u-beta-second-prop}. 
  Given (iii) of Proposition~\ref{prop:3} and Remark~\ref{rem:ubeta-vxi}, it remains only to show that
  $\Eps{\xi}{L_{t \wedge \sigma_{\R^\beta}}^x} \ge v^\xi(t,x) - u^{\beta}(t,x)$. We
  consider the alternative approximating sequence: $\tilde{\R}^m := \R^m \cap
  \R^\beta$. Recall from above that if $(t,x)\in \R^m$ then $\ovl t_{\R_m}(x)\leq \ovl t_{\R_{m+1}}(x)\le \ovl t^\beta(x)$ from which it follows that $\tilde{\R}^m$ is an
  increasing sequence of barriers. Moreover, from the definition of the points
  $x_m^k$, we have $\sigma_{\tilde{\R}^m} \searrow \sigma_{\R^\beta}$, since when we
  hit $\R^\beta$, we are guaranteed to hit $\tilde{\R}^m$ as soon as we have
  travelled at least $2^{-m+1}$ in both directions. However
  $\sigma_{\tilde{\R}^m} \ge \sigma_{\R^m}$, and therefore:
  \begin{equation*}
    \Eps{\xi}{L_{t \wedge \sigma_{\R^m}}^x} \le \Eps{\xi}{L_{t \wedge
        \sigma_{\tilde{\R}^m}}^x} \to \Eps{\xi}{L_{t \wedge \sigma_{\R^\beta}}^x}.
  \end{equation*}
  But also $\Eps{\xi}{L_{t \wedge \sigma_{\R^m}}^x} = v^\xi(t,x) -
  u^{\beta^m}(t,x) \to v^\xi(t,x) - u^\beta(t,x)$ and the result
  follows.
\end{proof}
We note that $\xi^\beta$-regularity of $\R^\beta$ now follows from Remark \ref{rem:regularity}.
The proof of Theorem \ref{thm:main_general_start} is complete.

\appendix

\section{Proofs of the optimality results}
\label{appendix:optimal}
 \setcounter{equation}{0}
\renewcommand {\theequation}{A.\arabic{equation}}

 We prove here the results announced in Section \ref{sec:optimality}. We start with establishing the required pathwise inequality.
\begin{proof}[Proof of Lemma \ref{lem:pathwiseineq}]
We proceed in three steps.
\\
{\bf 1.} We first observe that $\varphi_k\ge\varphi_{k+1}$ for all $k=1,\ldots,n$, and $\varphi_n=\varphi_{n+1}$ on $\R^n$. Indeed, notice that $\varphi_k(t,x)=\e^{t,x}[f(\zeta^{k})]$, where $\zeta^{k}$ is the first time we enter $\R^n$, having previously entered the barriers $\R^{n-1}, \R^{n-2}, \dots, \R^{k}$ in sequence. Then $\zeta^{k} \ge \zeta^{k+1}$, $\p^{t,x}$-a.s. implying that $\varphi_k\ge\varphi_{k+1}$ by the non-decrease of  $f$.
\\
{\bf 2.} We next compute that:
  \begin{eqnarray}\label{eq:diff_h_func}
    (h_k-h_{k-1})(t,x) - \lambda_{k-1}(x) 
    &=& 
    \int_t^{\tb^{k-1}(x)} (\varphi_{k-1}-\varphi_k)(s,x) ds.
  \end{eqnarray}
Then, $h_k-h_{k-1}-\lambda_{k-1}\ge 0$ for $t\le\tb^{k-1}(x)$, by Step 1. Next, notice that $t\ge\tb^{k-1}(x)$ if and only if $(s,x) \in \R^{k-1}$ for all $s\in[\tb^{k-1}(x),t]$, and that in this case $\sigma_{\R^{k-1}} = s, \p^{s,x}$-a.s., implying that $\varphi_{k-1}(s,x) = \varphi_k(s,x)$. Hence:
 \begin{eqnarray}\label{prop:GkIneq}
  h_k 
  \ge 
  h_{k-1} + \lambda_{k-1}
  &k=2,\dots,n,~\mbox{with equality on}&
   \R^{k-1}.
  \end{eqnarray}
{\bf 3.} By the previous steps, we have:
  \begin{align*}
    \sum_{i=1}^n   \lambda_i(x_i) 
    + \sum_{i=1}^n & \big[ h_i(s_i,x_i)-h_i(s_{i-1},x_{i-1})\big] + h_1(s_0,x_0) 
    \\
    &=\;
       \sum_{i=1}^n \lambda_i(x_i) 
       + \sum_{i=1}^{n-1} \big[h_i(s_i,x_i)-h_{i+1}(s_i,x_{i})\big] 
       + h_n(s_n,x_n) \\
      & \le\; \lambda_n(x_n) + h_n(s_n,x_n),
        ~~\mbox{with~ ``$=$''~ if}~(s_i,x_i)\in\R^i,~i=1,\ldots,n-1, 
      \\
      & =\;
        \int_0^{s_n} f(t)dt - \psi(x_n) f(0)
        - \int_{s_n}^{\tb^n(x_n)} (\varphi_n-\varphi_{n+1})(t,x_n) dt
       \\
       &\le\; 
       \int_0^{s_n} f(t)dt-\psi(x_n) f(0),
       ~~\mbox{with~ ``$=$''~ if}~(s_n,x_n)\in\R^n,
  \end{align*}
  where we used \eqref{eq:diff_h_func} and $\phi_{n+1}(x)=f(0)\int_0^x \eta(y)^{-2} \, dy$.
\end{proof}
To be able to take expectations in the pathwise inequality when applied to the stopped diffusion, we need to establish suitable (sub)martingale properties. These are isolated in the following lemma.
\begin{lemma}\label{lem:Gsubmart}
Let $f$ be bounded non-negative and non-decreasing, and assume 
 \begin{eqnarray}\label{opt:martcond}
 \int_0^. \phi_k(X_s)dX_s
 &\mbox{is a $\p^{\mu_0}-$martingale for all}&
 k=1,\ldots,n+1.
 \end{eqnarray}
Then, for all $k=1,\ldots,n$, the process $\{h_k(t,X_t)-h_k(0,X_0),t\ge 0\}$ is a $\p^{\mu_0}$-submartingale, and a $\p^{\mu_0}$-martingale on $\left[\sigma_{{k-1}}, \sigma_{k}\right]$.
\end{lemma}

\begin{proof}
First, applying the It\^o-Tanaka formula to the second term in the definition of $h_k$, we have
$$h_k(t,X_t) = h_k(0,X_0)+ \int_0^t \varphi_k(u,X_t)du - 2\int_0^t \phi_k(X_u)dX_u - \int_0^t \varphi_k(0,X_u)du,\quad t\geq 0. $$
Since $0\leq \varphi(u,x)\leq \|f\|_\infty<\infty$, \eqref{opt:martcond} shows that $h_k(t,X_t)-h_k(0,X_0)$ differs from a martingale by a bounded random variable and in particular is integrable.
We now proceed in two steps.
\\
{\bf 1.} For $0\le s\le t$, using the above decomposition and \eqref{opt:martcond}, we have
 \begin{eqnarray*}
 \e^{\mu_0}_s\big[h_k(t,X_t)\big] - h_k(0,X_0)
 &=&
 \int_0^t \e^{\mu_0}_s\big[\varphi_k(u,X_t)\big]du
 -\int_0^t \e^{\mu_0}_s\big[\varphi_k(0,X_u)\big]du - 2 \int_0^s \phi_k(X_u)dX_u
 \end{eqnarray*}
where $\e^{\mu_0}_s:=\e^{\mu_0}[.|\F_s]$. We shall prove in Step 2 below that
 \begin{eqnarray}
 \e^{\mu_0}_s\big[\varphi_k(u,X_t)\big]
 &\ge&
 \e^{\mu_0}_s\big[\varphi_k\big(u-(t-s),X_s\big)\big]
 ~~\mbox{for}~~
 u\in[t-s,t],
 \label{[t-s,t]}
 \\
  \e^{\mu_0}_s\big[\varphi_k(u,X_t)\big]
 &\ge&
 \e^{\mu_0}_s\big[\varphi_k(0,X_{t-u})\big]
 ~~\mbox{for}~~
 u\in[0,t-s],
 \label{[0,t-s]}
 \end{eqnarray}
and 
 \begin{eqnarray}\label{opt-equality}
 &\mbox{equality holds in \eqref{[t-s,t]}--\eqref{[0,t-s]} if}&
 \sigma_{k-1}\le s\le t\le \sigma_k.
 \end{eqnarray}
Then,
 \begin{eqnarray*}
 \e^{\mu_0}_s\big[h_k(t,X_t)\big]-h_k(0,X_0)
 &\ge&
 \int_0^{t-s} \e^{\mu_0}_s\big[\varphi_k(0,X_{t-u})\big]du
 +\int_{t-s}^t \e^{\mu_0}_s\big[\varphi_k(u-(t-s),X_s)\big]du\\
 &&-\int_0^t \e^{\mu_0}_s\big[\varphi_k(0,X_{u})\big]du - 2 \int_0^s \phi_k(X_u)dX_u
 \\
 &=&
 \int_s^{t} \e^{\mu_0}_s\big[\varphi_k(0,X_u)\big]du
 +\int_0^s \e^{\mu_0}_s\big[\varphi_k(u,X_s)\big]du\\
&& -\int_0^t \e^{\mu_0}_s\big[\varphi_k(0,X_{u})\big]du- 2 \int_0^s \phi_k(X_u)dX_u
 \\
 &=&
 h_k(s,X_s)-h_k(0,X_0)
 \end{eqnarray*}
with equality if $\sigma_{k-1}\le s\le t\le \sigma_k$.
\\
{\bf 2.} (i) We first argue, for all $(s,x)\in\re_+\times\re$, that
 \begin{equation}\label{varphi-submart}
 \big\{\varphi_k\big(t,X_{t}\big)\big\}_{t\ge s} 
 ~\mbox{is a submartingale on $[s,\infty)$, and a martingale on}~
  [s,\sigma_{\R^k}],~ \p^{s,x}-\mbox{a.s.}
 \end{equation} 
The martingale property is immediate from the definition of $\varphi_k$. The submartingale property follows from the following induction. First, the claim is obvious for $k=n+1$ by the fact that $f$ is non-decreasing. Next, suppose that the submartingale property in \eqref{varphi-submart} holds for some $k+1$. Introduce the stopping times $\sigma_{\R^k}^t:= \inf \{u \ge t: (u, X_u) \in \R^{k}\}$, and notice that $\sigma_{\R^k}^t \ge \sigma_{\R^k}^r$ for $s \le r \le t$. Then, denoting by $\tilde X, \tilde\sigma$ independent copies of the same objects, and using the induction hypothesis, we see that:
  \begin{eqnarray*}
    \Eps{(s,x)}{\varphi_k(t,X_t)| \F_r} 
    = 
    \Eps{(s,x)}{\Eps{(t,X_t)}{\varphi_{k+1}(\tilde{\sigma}_{\R^k},\tilde{X}_{\tilde{\sigma}_{\R^k}})}|\F_r}
    &=& 
    \Eps{(s,x)}{\varphi_{k+1}(\sigma_{\R^k}^t,X_{\sigma_{\R^k}^t})|\F_r}
    \\
    &\ge& 
    \Eps{(s,x)}{\varphi_{k+1}(\sigma_{\R^k}^r,X_{\sigma_{\R^k}^r})|\F_r} 
    \\
    &=&
    \varphi_k(r,X_r).
  \end{eqnarray*}
(ii) We now prove \eqref{[t-s,t]}. For $u \ge t-s$, it follows from \eqref{varphi-submart} that
  \begin{eqnarray}
    \e^{\mu_0}_s\big[\varphi_k(u,X_t)\big] 
    \;=\;
    \e^{0,X_s}\big[\varphi_k(u,\tilde{X}_{t-s})\big]
    &=& 
    \e^{u-(t-s),X_s}\big[\varphi_k(u,\tilde{X}_u)\big]
    \nonumber\\
    &=& 
    \e^{u-(t-s),X_s}\big[\varphi_{k+1}(\tilde \sigma_{\R^k}^u,\tilde{X}_{\tilde \sigma_{\R^k}^u})\big]
    \nonumber\\
    &\ge& 
    \e^{u-(t-s),X_s}\big[\varphi_{k+1}(\tilde \sigma_{\R^k},\tilde{X}_{\tilde\sigma_{\R^k}})\big]
    \label{eq:Mk_u_ge_t-s}\\
    &=& \varphi_k\big(u-(t-s),X_s\big),~~\p^{\mu_0}-\mbox{a.s.}
    \nonumber
  \end{eqnarray}
(iii) We next prove \eqref{[0,t-s]}. For $u\le t-s$, using again \eqref{varphi-submart}, we see that:
  \begin{eqnarray}
    \e^{\mu_0}_s\big[\varphi_k(u,X_t)\big] 
    &=& 
    \e^{\mu_0}_s\big[\e^{0,X_{t-u}}\big[\varphi_k(u,\tilde{X}_{u})\big]\big]
    \nonumber\\
    &\ge& 
    \e^{\mu_0}_s\big[\e^{0,X_{t-u}}\big[\varphi_k(0,\tilde{X}_{0})\big]\big]
    \label{eq:Mk_u_le_t-s}\\
    &=& 
    \e^{\mu_0}_s\big[\varphi_k(0,X_{t-u})\big].
    \nonumber
  \end{eqnarray}
(iv) Finally, to prove \eqref{opt-equality}, we observe that the equality was lost in \eqref{[t-s,t]} and \eqref{[0,t-s]} only because of the inequalities in \eqref{eq:Mk_u_ge_t-s} and \eqref{eq:Mk_u_le_t-s}, which in turn become equalities provided that $(u,X_u)$ does not enter $\R^k$ for $u \in [s,t)$. The condition that $\sigma_{k-1} \le s \le t \le \sigma_k$ ensures this is true.
\end{proof}

\begin{proof}[Proof of Theorem \ref{thm:optimal}]
Finally, we complete the proof of the main result in Section \ref{sec:optimality}.
First, by monotone convergence arguments and since $\psi$ is convex, note that 
\begin{equation}\label{eq:2ndmoment}
\Eps{\mu_0}{\rho_n}=\int \psi(x) (\mu_n-\mu_0)(dx)=\int (U^{\mu_0}(x)-U^{\mu_n}(x)) dx
\end{equation}
is the same for all $\rho \in \T(\bmu_n)$ so that adding a constant to $f$ does not change the problem. We shall normalise $f$ by taking $f(0)=0$ and exclude the trivial case $f\equiv 0$. If the quantities in \eqref{eq:2ndmoment} are equal to $+\infty$ then there is nothing to prove. We thus assume that \eqref{eq:2ndmoment} is finite. Note that this might be so even if $\int \psi(x) \mu_i(dx)=\infty$ for each $0\leq i\leq n$. More generally, thanks to the convex ordering of measures, one can define the integral $\int g(x)(\mu_j-\mu_i)(dx)$ for a convex $g$ and $0\leq i\leq j\leq n$. This is done by considering $g_k\nearrow g$ which are convex, equal to $g$ on a compact set and affine on the complement. Further, if $h = h-g + g$ with $(h-g)$ and $g$ convex with finite integrals against $(\mu_j-\mu_i)$ then the integral $\int h(x)(\mu_j-\mu_i)(dx)$ is also well defined and finite, see \cite{BNT:16} for details. We shall use this fact below repeatedly together with $\int \psi(x) (\mu_j-\mu_i)(dx)<\infty$ which follows from \eqref{eq:2ndmoment}. 

  Without loss of generality, we may assume that $f$ is bounded, the
  general case follows from a direct monotone convergence argument. Then $0 \le \varphi_i
  \le\|f\|_\infty$ for all $i$, and in particular, $|\phi_i(x)| \le
  \|f\|_{\infty} |\psi'(x)|$. We define $\kappa_i(x) := \int_0^x \phi_i(y)dy =
  -h_i(0,x)/2$, and observe that $\kappa_i(x)$ is then a non-negative,
  convex function with $0\leq \kappa_i''(x)\leq \|f\|_\infty \eta(x)^{-2}$ so that $\|f\|_\infty \psi(x) - \kappa_i(x)$ is a non-negative convex function. We conclude that $\int \kappa_i(x)(\mu_j-\mu_k)(dx)<\infty$, $0\leq k\leq j\leq n$.
  Moreover, we have $\kappa_i(x) \ge \kappa_{i+1}(x)$ for all $x \in \bI$ since $\varphi_i\geq \varphi_{i+1}$, as argued above. 
  
  The aim is now to take expectations in \eqref{eq:pathwiseineq} for
  $(s_i, x_i) = (\rho_i,X_{\rho_i})$, where $\rho \in \T(\bmu_n)$. To do this, we need to check
  that the expectations under $\p^{\mu_0}$ of individual terms on the right-hand side of
  \eqref{eq:pathwiseineq} are well defined. 
  
  We can rewrite the first two terms on the right-hand side of \eqref{eq:pathwiseineq} as:
  \begin{align*}
    \sum_{i=1}^n \lambda_i(x_i) 
    +h_1(0,x_0) & = \sum_{i=1}^n \int_0^{\ovl t_i(x_i)} \left( \varphi_{i+1}(s,x_i)-\varphi_i(s,x_i)\right)ds + 2 \sum_{i=1}^n \int_{x_{i-1}}^{x_{i}}\phi_{i}(y) dy,
  \end{align*}
where we used that $f(0)=0$ so that $\kappa_{n+1}\equiv 0$ and $h_{n+1}(t,x)=\int_0^t f(u)du$. The expectation of the first two terms is then equal to
  \begin{equation*}
    \sum_{i=1}^n \int \int_0^{\ovl t_i(x)}(\varphi_{i+1}(s,x)-\varphi_i(s,x))ds\,\mu_i(dx) + 2 \sum_{i=1}^n \int \kappa_i(x)(\mu_i-\mu_{i-1})(dx).
\end{equation*}
The integrals in the second sum are well defined and finite by the discussion above. As for the first sum, observe that
$$ 0\leq \int \int_0^{\ovl t_i(x)}(\varphi_{i}(s,x)-\varphi_{i+1}(s,x))ds\,\mu_i(dx)\leq \|f\|_\infty \Eps{\mu_0}{\ovl t_i(X_{\sigma_i})}\leq \|f\|_\infty \Eps{\mu_0}{\sigma_i}<\infty.
$$
Using $|\varphi_i|\leq \|f\|_\infty$, $\Eps{\mu_0}{\rho_n}<\infty$ and integrability properties of $\kappa_i$ we see that the local martingale 
$$\int_0^t \phi_i(X_u)dX_u = \kappa_i(X_t)-\kappa_i(X_0)-\frac{1}{2}\int_0^t \varphi_i(0,X_u)du,$$
is a martingale on $[0,\rho_n]$. It then follows from Lemma \ref{lem:Gsubmart} that
 $$ \Eps{\mu_0}{h_i(\rho_i,X_{\rho_i})-h_i(\rho_{i-1},X_{\rho_{i-1}})}\geq 0, \quad i=1,\ldots, n,$$
with equality if $\rho_i=\sigma_i$. Taking
  expectations under $\p^{\mu_0}$ in \eqref{eq:pathwiseineq}, we deduce that
\begin{equation*}
  \Eps{\mu_0}{\int_0^{\rho_n}f(t)\, dt} \ge \sum_{i=1}^n \int \int_0^{\ovl t_i(x)}(\varphi_{i+1}(s,x)-\varphi_i(s,x))ds\,\mu_i(dx) + 2 \sum_{i=1}^n \int \kappa_i(x)(\mu_i-\mu_{i-1})(dx),
\end{equation*}
with equality when we replace $\rho_n$ with $\sigma_n$.
\end{proof}

\section{Extension to continuous Markov local martingales} \label{ap:MarkovMartingales}
\setcounter{equation}{0}
\renewcommand {\theequation}{B.\arabic{equation}}

The following statement extends Lemma \ref{lem:box} to a class of continuous Markov local martingales.

\begin{lemma}\label{lem:boxMarkov}
Let $X$ be a local martingale with $d\langle X\rangle_t=\eta(X_t)^2dt$, for some locally Lipschitz function $\eta$, and let $a<b$ be fixed points in $\iI$, and $H_{a,b}$ the first exit time of $X$ from the interval $(a,b)$. Then
 \begin{eqnarray*}
 \e^x\big| X_{t\wedge H_{a,b}}-y\big|
 \;=\;
 \e^y\big| X_{t\wedge H_{a,b}}-x\big|
 &\mbox{for all}&
 x,y\in[a,b].
 \end{eqnarray*} 
\end{lemma}

\begin{proof}
Le $y\in(a,b)$ be fixed, and denote $X^H:=X_{.\wedge H_{a,b}}$. We decompose the proof in three steps.
\\
{\it Step 1:} By dominated convergence the function $u(t,x):=\e^x\big| X^H_t-y\big|$ is continuous, and it follows from classical argument using the tower property that $u$ is a viscosity solution of the equation
 \begin{equation}\label{heatequationab}
 \begin{array}{c}
 \big(\partial_t u - \frac12 \eta^2D^2u\big)(t,x)
 =
 0
 ~~\mbox{for}~~
 t\ge 0,~x\in(a,b)
 \\
 u(x,a)=y-a,~~u(x,b)=b-y,~x\in(a,b).
 \end{array}
 \end{equation}
{\it Step 2:} Similarly, the function $v(t,x):=\e^y\big| X^H_t-x\big|$ is a continuous function, and is in addition convex in the $x-$variable. Denote by $L(X^H)$ the local time of the continuous martingale $X^H$. Using the It\^o-Tanaka formula, we see that:
 \begin{eqnarray*}
 v(t+h,x)-v(t,x)
 &=&
 \e^y \big[L^x_{t+h}(X^H)-L^x_{t}(X^H)\big].
  \end{eqnarray*}
By the density occupation formula, this provides for all Borel subset $A$ of $[a,b]$:
 \begin{eqnarray*}
 \int_A \int_t^{t+h} \partial_tv(ds,x)dx
 &=&
 \int_A \big(v(t+h,x)-v(t,x)\big)dx
 \;=\;
 \int_A\eta^2(x)\int_t^{t+h} \mathbb{P}^{X^H_s}(dx)ds,
 \end{eqnarray*}
where $\mathbb{P}^{X^H_s}$ denotes the distribution function of $X^H_s$. Notice that $\mathbb{P}^{X^H_s}=\frac12 D^2v(s,.)$. Then:
 \begin{eqnarray*}
 \int_A \int_t^{t+h} \partial_tv(ds,x)dx
 &=&
 \int_A\int_t^{t+h}\frac12\eta^2(x) ds D^2v(s,dx).
 \end{eqnarray*}
Let $\varphi_\varepsilon$ be a $C^\infty-$molifier, and set $v_\varepsilon(t,x)=\int v(t-s,x-y)\varphi_\varepsilon(s,y)ds dy$. Then, $v_\varepsilon$ is smooth, and it follows from the last equality that
 \begin{eqnarray*}
  \int_A \int_t^{t+h} 
  \big(\partial_tv_\varepsilon-\frac12\eta^2D^2v_\varepsilon-R_\varepsilon\big)(s,x)ds dx
 &=&
 0,
 \end{eqnarray*}
where $R_\varepsilon(s,x):=\int \big(\eta^2(x)-\eta^2(x-y)\big)D^2v(r-s,x-y)\varphi_\varepsilon(r,y)dr dy$. Since $\eta$ is Lipschitz on $[a,b]$, and $v$ is bounded, we see that
 \begin{eqnarray*}
 \big|R_\varepsilon(s,x)\big|
 &\le&
 c\int D^2\{|x-y|\varphi_\varepsilon(r-s,x-y)\}dr dy
 \\
 &=&
 c\int \big[D\{|x-y|\varphi_\varepsilon(r-s,x-y)\}\big]_a^b dr
 \;=:\; r_\varepsilon \;\longrightarrow\; 0,~~\mbox{as}~~\varepsilon \to 0.
 \end{eqnarray*}
By the arbitrariness of $h>0$ and the Borel subset $A$ of $[a,b]$, this shows that 
 \begin{eqnarray*}
 \partial_tv_\varepsilon
 -\frac12\eta^2D^2v_\varepsilon-r_\varepsilon\ge 0
 &\mbox{and}&
 \partial_tv_\varepsilon
 -\frac12\eta^2D^2v_\varepsilon+r_\varepsilon\le 0
 ~~\mbox{on}~~
 \re_+\times (a,b).
 \end{eqnarray*}
Since $v_\varepsilon\longrightarrow v$, locally uniformly, it follows from the stability result of viscosity solutions that $v$ is a viscosity solution of $\partial_tv-\frac12\eta^2D^2v=0$ on $\re_+\times (a,b)$. We also directly see that $v(t,a)=y-a$ and $v(t,b)=b-y$. Hence $v$ is also a viscosity solution of \eqref{heatequationab}.
\\
{\it Step 3:} To conclude that $u=v$, we now use the fact that equation \eqref{heatequationab} has a unique $C^0(\re_+\times[a,b])$ viscosity solution. Indeed the corresponding equation satisfied by $e^{\lambda t} u(t,x)$, for an arbitrary $\lambda>0$, satisfies the conditions of Theorem 8.2 of \cite{CrandallIshiiLions}. 
\end{proof}

\bibliographystyle{apalike}

\bibliography{bib_OSPRoot}

\end{document}